\documentclass[a4paper]{article}
\usepackage[english]{babel}

\usepackage{amssymb, amsmath,amsthm}
\usepackage{mathrsfs}
\usepackage{amscd}
\usepackage[active]{srcltx}
\usepackage{verbatim}
\usepackage{graphicx}

\usepackage[utf8]{inputenc}
\usepackage[T1]{fontenc}
\usepackage{lmodern}

\usepackage{hyperref} 
\usepackage{todonotes}
\usepackage{enumerate}
\usepackage{mathtools}
\usepackage{bbm}
\usepackage[normalem]{ulem}

\mathtoolsset{showonlyrefs,showmanualtags}

\newcommand\dd{\mathrm{d}}
\newcommand\ud{\,\mathrm{d}}

\newcommand\EE{\mathbb{E}}
\newcommand\Dd{\mathcal{D}}
\newcommand\Ss{\mathcal{S}}
\newcommand\Nn{\mathcal{N}}
\newcommand\Hh{\mathcal{H}}

\newcommand\Ff{\mathcal{F}}

\newcommand{\RR}{\mathbb{R}}

\newcommand{\NN}{\mathbb{N}}

\newcommand{\PP}{\mathbb{P}}

\newcommand{\inp}[2]{\langle #1,#2 \rangle}

\newcommand{\Ric}{\mathrm{Ric}}
\newcommand{\Rm}{\mathrm{Rm}}
\newcommand{\Nabla}{\nabla}
\newcommand{\Exp}{\mathrm{Exp}}

\newcommand{\cA}{\mathcal{A}}

\newcommand{\cC}{\mathcal{C}}
\newcommand{\cD}{\mathcal{D}}

\newcommand{\cH}{\mathcal{H}}
\newcommand{\cI}{\mathcal{I}}

\newcommand{\cL}{\mathcal{L}}

\newcommand{\cP}{\mathcal{P}}

\newcommand{\cS}{\mathcal{S}}

\newcommand{\cV}{\mathcal{V}}

\newcommand{\bE}{\mathbb{E}}

\newcommand{\bP}{\mathbb{P}}

\newcommand{\bR}{\mathbb{R}}

\newcommand{\bfH}{\mathbf{H}}

\newcommand{\bfR}{\mathbf{R}}

\newcommand{\bfV}{\mathbf{V}}

\DeclareMathOperator*{\LIM}{LIM}

\newcommand{\sn}[1]{\left|  #1 \right|} 
\newcommand{\vn}[1]{\left| \! \left| #1\right| \! \right|} 

\newcommand{\ip}[2]{\langle #1,#2\rangle}

\renewcommand{\epsilon}{\varepsilon}

\theoremstyle{plain}
\newtheorem{theorem}{Theorem}[section]
\theoremstyle{definition}
\newtheorem{example}[theorem]{Example}
\theoremstyle{remark}
\newtheorem{remark}[theorem]{Remark}
\theoremstyle{plain}

\newtheorem{lemma}[theorem]{Lemma}
\newtheorem{proposition}[theorem]{Proposition}
\newtheorem{definition}[theorem]{Definition}

\newtheorem{assumption}[theorem]{Assumption}

\numberwithin{equation}{section}

\allowdisplaybreaks

\begin{document}

\title{Classical large deviations theorems on complete Riemannian manifolds}

\author{
\renewcommand{\thefootnote}{\arabic{footnote}}
Richard C. Kraaij\footnotemark[1],
~Frank Redig\footnotemark[2],
~Rik Versendaal\footnotemark[2]
}

\footnotetext[1]{
Fakult\"at f\"ur Mathematik, Ruhr-University of Bochum, Postfach 102148, 
44721 Bochum, Germany, E-mail: \texttt{richard.kraaij@rub.de}.
}

\footnotetext[2]{
Delft Institute of Applied Mathematics, Delft University of 
Technology, P.O. Box 5031, 2600 GA Delft, The Netherlands, E-mail: \texttt{F.H.J.Redig/R.Versendaal@tudelft.nl}.
}

\date\today


\maketitle

\begin{abstract}
We generalize classical large deviations theorems to the setting of complete Riemannian manifolds. We prove the analogue of Mogulskii's theorem for geodesic random walks via a general approach using visocity solutions for Hamilton-Jacobi equations. As a corollary, we also obtain the analogue of Cram\'er's theorem. The approach also provides a new proof of Schilder's theorem. Additionally, we provide a proof of Schilder's theorem by using an embedding into Euclidean space, together with Freidlin-Wentzell theory.\\

\textit{keywords}: large deviations, Cram\'er's theorem, geodesic random walks, Riemannian Brownian motion, non-linear semigroup method, Hamilton-Jacobi equation
\end{abstract}


\tableofcontents

\section{Introduction}\label{section:Introduction}

In the theory of large deviations, a fundamental result is Cram\'er's theorem (see e.g. \cite{DZ98,Hol00}), stating that the empirical mean of $n$ independent identically
distributed random variables
satisfies the large deviation principle. The large deviation principle is intuitively stated as
$$
\PP\left(\frac1n\sum_{i=1}^n X_i \approx x\right) \approx e^{-nI(x)}.
$$
Here the rate function $I$ is the Legendre transform of the log moment generating function, i.e.,
\begin{equation}\label{iii}
I(x)= \sup_t \{\langle t, x\rangle -\Lambda(t)\},
\end{equation}
where $\Lambda(t)= \log \EE(e^{\langle t,x\rangle})$.
Cram\'er's theorem holds in a very general setting and is also the starting point of several other large deviation results.
There are also generalizations weakening the assumption of independence, e.g. the G\"artner-Ellis theorem.
Furthermore, there are various path space large deviation results which have Cram\'er's theorem as a starting point.
For example, from Cram\'er's theorem in the Banach space setting one can derive Schilder's theorem for
path space large deviations of rescaled Brownian motion (see e.g. \cite{DS89,DZ98,Str84}).

There is a natural path space large deviation result which accompanies Cram\'er's theorem, namely Mogulskii's theorem. It states that
in the same setting, the random paths $t\in [0,1]\mapsto S_n(t)$ where
$$
S_n(t)= \frac1n\sum_{i=1}^{\lfloor nt\rfloor} X_i
$$
satisfy the large deviation theorem in the sense
$$
\PP\left( S_n(\cdot) \approx \gamma \right)\approx e^{-n \cI (\gamma)}
$$
where
$$
\cI(\gamma)= \int_0^1 I(\dot{\gamma_t}) dt.
$$
Here, $\dot{\gamma_t}$ denotes the velocity of the path $\gamma$ at time $t$, and $I$ is the rate function
\eqref{iii}.
In the proof of this theorem, Cram\'er's theorem is the starting point because it gives the large deviations
for the finite-dimensional marginals of $S_n(\cdot)$.
To lift this result to the path space large deviations in the weakest topology, one can rely on the framework of projective limits
(Dawson-G\"artner theorem).
Finally, to pass to a stronger topology such as the uniform topology, one has to prove exponential tightness in the chosen topology.

On the other hand, once the large deviation principle
for the trajectories $S_n(\cdot)$ is obtained, one can of course re-obtain Cram\'er's theorem by putting $t=1$, and by applying the contraction principle.\\

It is a natural question to ask how in particular Cram\'er's large deviation theorem is generalized to the setting of a Riemannian manifold. The main obstacle is that the manifold itself has no additive structure, and therefore a random walk cannot be defined as an addition of independent increments. This problem can be tackled in the spirit of the paper of Jorgenson (\cite{Jor75}), by introducing an appropriate family $\{\mu_x\}_{x\in M}$ of probability measures (or equivalently, random variables $\{X_x\}_{x\in M}$) on the tangent spaces $T_xM$ at points $x\in M$. The summing of independent increments is then replaced by an iterative application of the exponential map to the increment on the tangent space of the point where the random walk has arrived. More precisely, the analogue of normalized sum $\frac1n \sum_{i=1}^n X_i$ is build via the recursion $A_0=x_0\in M$, and
\begin{equation}\label{rawa}
A_{i+1}= \Exp_{A_i}\left(\frac1n X_{i+1}\right)
\end{equation}
for $i=0,1,2,\ldots, n-1$. The random variable $A_n$ then takes values in $M$ and is the natural analogue of the empirical average $\frac1n \sum_{i=1}^n X_i$.
In the case of flat space, geodesics are straight lines, and as a consequence $\Exp_x(tv) = x + tv$. Therefore, in that case, $A_n = \frac1n\sum_{i=1}^nX_i$.
In general however, due to curvature, $A_n$ is a complicated function of the increments, which is not even permutation invariant.
Nevertheless, purely working via analogy, one can make a reasonable guess for what the large deviations of $A_n$ should look like.
Define first the analogue of the log-moment generating function:
$$
\Lambda_{\mu_x} (\lambda)= \log \EE\left(e^{\langle X_x, \lambda \rangle}\right)
$$
where $\lambda$ is now an element of the cotangent space at $x\in M$, denoted by $T^*_xM$.
By imposing appropriate invariance properties of the family $\{X_x\}_{x\in M}$, this function
satisfies
$$
\Lambda_x (\lambda)= \Lambda_y (\tau_{xy} \lambda)
$$
where $\tau_{xy}$ denotes parallel transport of the form $\lambda\in T^*_xM$ to $T^*_yM$ along any smooth curve connecting $x$ and $y$.
The  natural candidate rate function is then $I_M(x)$ (the subindex $M$ referring to the manifold)
$$
I_M(x)=\inf_{v\in \cV(x_0,x)} \sup_{\lambda\in T^*_p M} \left( \langle v,\lambda \rangle \right)-\Lambda_p (\lambda)
$$
where the first infimum is over the set of all initial velocities of geodesics leading from $p$ to $x$ in time $1$.
Notice that in the case of flat space, with $p=0$, the only possible $v$ is precisely $x$, which is the speed
of the geodesic -a straight line- leading from $0$ to $x$ in time $1$, so 
in that case $I_M(x)$ coincides with \eqref{iii}.

One of the results of our paper is that $I_M(x)$ is indeed the correct rate function for the large deviations of the averages $A_n$.
Somewhat surprisingly, in order to obtain this result, one first needs the analogue of Mogulskii's theorem, which in turn can be obtained
by the robust method of path-space large deviations for sequences of Markov processes from \cite{FK06} (here named ``the Feng-Kurtz formalism'').
Indeed, the recursion
defines a discrete-time Markov process which has an $n$-dependent transition operator, which puts us precisely in the realm of
\cite{FK06}.\\

Additionally, we show that the Feng-Kurtz method can be used to given a new proof of Schilder's theorem for Riemannian Brownian motion. In \cite{Var67}, Varadhan studied the short time behaviour of the heat kernel associated to Riemannian Brownian motion and proved that
$$
\lim_{t\to0} t\log p_t(x,y) =  -\frac{d(x,y)^2}{2}
$$
where $d$ is the Riemannian distance on the manifold. Afterwards, analogues of Schilder's theorem for Euclidean Brownian motion, the large deviations for Riemannian Brownian motion have been studied, and can be found in e.g. \cite{Aze78,FW12}.\\

To our knowledge, the generalizations of Mogulskii's and Cram\'er's theorem to geodesic random walks on a Riemannian manifold are new results. Additionally, the approach of using the Feng-Kurtz formalism in the setting of Riemannian geometry is novel and of independent interest.

The paper is organised as follows. In Section \ref{section:Notation} we introduce some basic notions from differential geometry, as well as from large deviations theory. In Section \ref{section:Brownianmotion} we provide a review of the construction of Riemannian Brownian motion and collect important results we need in what follows. In Section \ref{section:geodesicRW} we introduce the geodesic random walks we need for the analogue of Mogulskii's and Cram\'er's theorem. Then, in Section \ref{section:mainresults} we state the main large deviation results. A new proof of Schilder's theorem using embeddings can be found in Section \ref{section:Schilderembedding}. In Section \ref{section:abstract_HJ_to_LDP} we introduce the Feng-Kurtz formalism and show how this is applicable in the Riemannian setting. Finally, with the main work done, in Section \ref{section:resultsFK} we provide the proofs of the theorems stated in Section \ref{section:mainresults} using this Feng-Kurtz formalism.

\section{Notation and important notions}\label{section:Notation}

In this section we will introduce some basic notation and concepts from differential geometry (see e.g. \cite{Lee97,Spi79}), as well as the definition of the large deviation principle (see e.g. \cite{DZ98}). We conclude the section by introducing Freidlin-Wentzell theory in Euclidean space.

\subsection{Some differential geometry}\label{subsection:geometry}

Throughout this paper we work in a complete Riemannian manifold $(M,g)$ of dimension $k$.  We denote by $TM$ the tangent bundle, and by $T^*M$ the cotangent bundle. By $\Gamma(TM)$ we denote the vector fields, i.e., the smooth sections of $TM$. With the idea of studying Hamiltonians in mind, we reserve $p$ for elements of $T^*M$ as momentum is a cotangent vector. Tangent vectors are generally denoted $V$ and we write $x$ for points in $M$. 

\subsubsection{Riemannian inner product and distance}

For $x \in M$, if $V \in T_xM$ and $p \in T_x^*M$, we denote the duality pairing by $\inp{V}{p}$ or $p(V)$. The inner product of two elements $V,W \in T_xM$ is denoted by $\inp{V}{W}_{g(x)}$ and the length $|V|_{g(x)}$ of $V$ is defined as 
$$
|V|_{g(x)} = \sqrt{\inp{V}{V}_{g(x)}}.
$$

The \emph{Riemannian distance} on $M$ is defined as
$$
d(x,y) := \inf\left\{\int_0^1 |\dot\gamma(t)|_{g(\gamma(t))}\ud t \, \middle| \, \gamma:[0,1] \to M, \gamma(0) = x, \gamma(1) = y, \gamma \mbox{ piecewise smooth}\right\}.
$$

\smallskip

We define the inner product, and consequently the length, of cotangent vectors using duality via the metric. For $p \in T_x^*M$, we will denote its length by $|p|_{g(x)}$. We will omit the point $x$ in the notation whenever this is clear. In a coordinate chart, we denote the coefficient matrix of the metric by $G = (g_{ij})$. The coefficient matrix of the metric on cotangent vectors is then given by $G^{-1} = (g^{ij})$.

\subsubsection{Connection and parallel transport}

We assume that our Riemannian manifold is equipped with the \emph{Levi-Civita connection} $\Nabla$, i.e., the unique connection which is compatible with the metric and torsion free. 

A vector field $V \in \Gamma(TM)$ is parallel along a curve $\gamma:[a,b] \to M$ if $\Nabla_{\dot\gamma(t)}V(\gamma(t)) = 0$ for all $t \in [a,b]$. A curve $\gamma$ is called a geodesic, if the vector field $\dot\gamma(t)$ is parallel along $\gamma$. It turns out that paths of minimal length between points are geodesic for the Levi-Civita connection.

A connection induces the notion of \emph{parallel transport}. Given a (piecewise) smooth curve $\gamma:[a,b] \to M$, we denote parallel transport along $\gamma$ from $\gamma(t_0)$ to $\gamma(t_1)$ by $\tau_{\gamma,t_0t_1}$, or simply $\tau_{t_0t_1}$ whenever the meant curve is clear. If points $x,y \in M$ can be connected by a unique geodesic of minimal length, we will also write $\tau_{xy}$ meaning parallel transport from $x$ to $y$ along this specific geodesic. To define parallel transport for cotangent vectors, one can use the duality between the tangent space and cotangent space. In particular, if $p \in T_x^*M$ and $V \in T_yM$, then $\tau_{xy}p(V) = p(\tau_{xy}^{-1}V)$. This also characterizes $\tau_{xy}p$ if this is satisfied for all $V \in T_yM$.

The following proposition is used in proving the generalization of Mogulskii's theorem. For a proof, we refer to Appendix \ref{section:proofs_propositions_lemmas}.

\begin{proposition}\label{proposition:transport_derivative_distance}
Let $x,y \in M$ and assume that $x \notin C_y$ (or equivalently, $y \notin C_x)$. Then for all $V \in T_yM$ we have
$$
\dd_yd^2(x,y)(V) = 2\inp{\dot\gamma(1)}{V}_{g(y)},
$$
where $\gamma:[0,1]\to M$ is the unique geodesic of minimal length connecting $x$ and $y$. Consequently, we obtain
$$
\tau_{xy}\dd_xd^2(x,y) = -\dd_yd^2(x,y).
$$
\end{proposition}

\subsubsection{Exponential map, injectivity radius and cutlocus}

Given $x \in M$, for every $V \in T_xM$, let $\gamma_V$ be the geodesic starting at $\gamma_V(0) = x$, with initial velocity $V$. By the completeness assumption on $M$, this geodesic exists for all times $t > 0$. We define the exponential map $\Exp_p: T_pM \to M$ by setting $\Exp_x(V) = \gamma_V(1)$. A geodesic ball $B(x,\delta)$ of radius $\delta > 0$ in $M$ is the image of $\{V \in T_xM \, | \, |V|_{g(x)} < \delta\}$ under $\Exp_p$. A geodesic sphere of radius $\delta > 0$ is the image of $\{V \in T_xM \, | \, |V|_{g(x)} = \delta\}$.

For a point $x \in M$, we define the \emph{injectivity radius} at $x$ to be
$$
i(x) := \sup\left\{\delta \in \RR_+ \, \middle| \, \Exp_x \mbox{ is injective on } B_x(0,\delta)\right\},
$$
where $B_x(0,\delta) \subseteq T_xM$ is the ball of radius $\delta$ with respect to $g(x)$. The existence of normal coordinates around $x$ assures that $i(x) > 0$. Observe that on $\Exp_x(B_x(0,i(x)))$ the map $y \mapsto d^2(x,y)$ is smooth.  We define the injectivity radius of the manifold to be
$$
i(M) := \inf_{x\in M} i(x)
$$
\begin{proposition}[\cite{Kli82}]\label{prop:injectivityradius}
The injectivity radius $i(x)$ depends continuously on $x$. In particular, if $M$ is compact we have $i(M) > 0$.
\end{proposition}

The injectivity radius is closely related to the cutlocus of a point $x \in M$. For any $x \in M$ we define the \emph{cutlocus} $C_x$ to be the set of all point $y \in M$ for which there is more than one geodesic of minimal length connecting $x$ and $y$. 

\subsubsection{Curvature}

The \emph{Riemann curvature endomorphism} measures to what extent second order covariant derivatives of a vector field commute. It is the map $R: TM \times TM \times TM \to TM$ given by
$$
R(X,Y)Z = \Nabla_X\Nabla_YZ - \Nabla_Y\Nabla_XZ - \Nabla_{[X,Y]}Z,
$$
where $[X,Y] = XY - YX$ is the commutator of $X$ and $Y$. Associated to this is the \emph{Riemann curvature tensor}, which is the 4-tensor $\Rm$
$$
\Rm(X,Y,Z,W) = \inp{R(X,Y)Z}{W}.
$$
Finally, by taking the trace of the curvature tensor with respect to the first and last entry, we obtain a 2-tensor which we will call the Ricci-curvature, denote by $\Ric$.

\subsubsection{Function spaces}

We denote the space of continuous functions on $M$ by $C(M)$ and the space of bounded continuous functions by $C_b(M)$. The smooth functions are indicated by $C^\infty(M)$, whereas the space of smooth functions that are constant outside of a compact set are denoted by $C^\infty_c(M)$.

The set of continuous curves on an interval $I \subseteq \bR^+$ is denoted by by $C(I;M)$. Spaces of continuous functions and curves are considered to carry the topology of uniform convergence. We denote the Skorokhod space of cádlág paths by $D(I;M)$, see \cite[Section 3.5]{EK86}. Finally, we define the space $H^1(I,M)$ by
$$
H^1(I;M) := \left\{\gamma:I \to M|\gamma \mbox{ is differentiable a.e. and } \int_I |\dot\gamma(t)|_g^2 \ud t < \infty\right\}
$$
with norm given by
$$
||\gamma||_{H^1} = \int_I |\dot\gamma(t)|^2_g\ud t.
$$

Finally, we denote by $\cA\cC(I;M)$ the set of absolutely continuous curves in $M$, i.e. the set of continuous curves $\gamma : I \rightarrow M$ that differentiable for almost every point in $I$ and such that for all $f \in C^\infty(M)$ and $(s,t) \subseteq I$:
\begin{gather*}
\int_s^t |\ip{\dd f(\gamma(r))}{\dot{\gamma}(r)}| \dd r < \infty, \\
f(\gamma(t)) - f(\gamma(s))  = \int_s^t \ip{\dd f(\gamma(r))}{\dot{\gamma}(r)} \dd r.
\end{gather*}

If we consider only curves with some fixed initial point $x \in M$, we write $C_x(I;M), H^1_x(I;M)$ and $\cA\cC_x(I;M)$ to indicate this.

\subsection{Large deviation principle}\label{subsection:LDP}

Large deviation principles control the limiting behaviour on the exponential scale of a sequence $\mu_n$ of probability measures on some state space $\Omega$. This limiting behaviour is governed by a rate function $I$, which is a lower semi-continuous function from $\Omega$ into $[0,\infty]$. We say that $I$ is a good rate function if its  sublevel sets $\{\omega \, | \, I(\omega) \leq c\}$ are compact. 

\begin{definition}
	Consider a sequence of measures $\{\mu_n\}_{n\geq 1}$ on $\Omega$.
	
	\begin{enumerate}[(a)]
		\item We say that the sequence $\{\mu_n\}_{n\geq 1}$  is \textit{exponentially tight} if for all $\alpha > 0$ there exists a compact $K_\alpha \subseteq \Omega$ such that
		\begin{equation*}
		\limsup_{n \to \infty} \frac1n \log\mu_n(K_\alpha^c) < -\alpha.
		\end{equation*}
		\item We say that the sequence $\{\mu_n\}_{n\geq 1}$ satisfies the \textit{large deviation principle} (LDP) on $\Omega$ with rate function $I : \Omega \rightarrow [0,\infty]$ if it satisfies
		\begin{enumerate}[(i)]
			\item \textit{The upper bound}; for every closed set $F \subseteq \Omega$,
			$$
			\limsup_{n\to\infty} \frac1n\log\mu_n(F) \leq -\inf_{\omega \in F} I(\omega).
			$$
			\item \textit{The lower bound}; For every open set $G \subseteq \Omega$,
			$$
			\liminf_{n\to\infty} \frac1n\log\mu_n(G) \geq -\inf_{\omega \in G} I(\omega).
			$$
		\end{enumerate}
	\end{enumerate}
Depending on the situation, we will also write $\epsilon$ instead of $\frac 1n$ and let $\epsilon$ tend to 0.
\end{definition}

\subsection{Freidlin-Wentzell theory}\label{subsection:FreidlinWentzell}


We conclude this introductory section with a short discussion of Freidlin-Wentzell theory in Euclidean space, which we will use in the proof of Schilder's theorem for Riemannian Brownian motion. The theory of Freidlin and Wentzell is concerned with LDPs for solutions $X_t^\epsilon$ of stochastic differential equations of the form
\begin{equation}\label{eq:FW}
\dd X_t^\epsilon = b(X_t^\epsilon)\dd t + \sqrt{\epsilon}\sigma(X_t^\epsilon)\dd W_t
\end{equation}
where $W_t$ is a $\RR^l$-valued Brownian motion and $b:\RR^k \to \RR^k$ and $\sigma:\RR^k \to \RR^{k\times l}$. We have the following theorem (see \cite[Theorem 5.6.7]{DZ98}, (combined with \cite[Theorem 5(a)]{Big04}).

\begin{theorem}[Freidlin-Wentzell]\label{theorem:FWIto}
Assume that $\{X_0^\epsilon\}$ satisfies the LDP in $\RR^k$ with good rate function $I_0$. If the entries of $b$ and $\sigma$ are bounded, Lipschitz continuous, then the solution $X_t^\epsilon$ of $\eqref{eq:FW}$ satisfies the LDP in $C(\RR^+;\RR^k)$ with the good rate function
\begin{equation}\label{equation:RFFW}
\begin{split}
I(f) = \inf\Bigg\{I_0(f(0)) + \frac12\int_0^\infty |\dot g(t)|^2\ud t \Bigg| g \in \cA\cC(\RR^+;\RR^k), \qquad \qquad \\f(t) = x + \int_0^t b(f(s))\ud s + \int_0^t \sigma(f(s))\dot g(s)\ud s\Bigg\}
\end{split}
\end{equation}
\end{theorem}

As on manifolds we will be working with Stratonovich stochastic differential equations instead of the It\^o ones as above, we need the following adjustment of the theorem. The proof can be found in Appendix \ref{section:proofs_propositions_lemmas}.

\begin{theorem}[Freidlin-Wentzell, Stratonovich SDE]\label{thm:FW}
Let $W_t$ be an $\RR^l$ valued standard Brownian motion. Let $b:\RR^k \to \RR^k$ and $\sigma:\RR^k \to \RR^{k\times l}$ be bounded, Lipschitz continuous functions. Assume that for any $\epsilon > 0$ the process $Y_t^\epsilon$ satisfies the Stratonovich stochastic differential equation
\begin{equation}\label{eq:SDEFW}
\dd Y_t^\epsilon = b(Y_t^\epsilon)\ud t + \sqrt{\epsilon}\sigma(Y_t^\epsilon)\circ\dd W_t,
\end{equation}
and assume that $\{Y_0^\epsilon\}$ satisfies the LDP in $\RR^k$ with good rate function $I_0$. Then the trajectories of $Y_t^\epsilon$ satisfy the LDP in $C(\RR^+;\RR^k)$ with good rate function as in \eqref{equation:RFFW}.
\end{theorem}


\section{Brownian motion on Riemannian manifolds}\label{section:Brownianmotion}

In this section, we give a concise review of the definition of Brownian motion on a Riemannian manifold, following \cite{Hsu02}. We go over two equivalent definitions which we need for the different approaches to proving Schilder's theorem in Sections \ref{section:Schilderembedding} and \ref{section:resultsFK}. We end with a short discussion on the behaviour of the radial process. Although this section has a review character, we made it as self-contained as possible. The reader familiar with the various definitions of Riemannian Brownian motion might skip this section.

\subsection{Generator approach to Brownian motion}

Denote by $\Delta_M$ the Laplace-Beltrami operator on $M$. On any coordinate chart $(x,U)$ we have
$$
\Delta_M = \frac{1}{\sqrt{\det G}}\frac{\partial}{\partial x^i}\left(\sqrt{\det G}g^{ij}\frac{\partial}{\partial x^j}\right).
$$

In the Euclidean case, the generator of Brownian motion is given by $\frac12\Delta$. This inspires the following definition.

\begin{definition}[Riemannian Brownian motion]
A continuous $M$-valued process $W_t$ is a \emph{Riemannian Brownian motion} if it is generated by $\frac12\Delta_M$, i.e., $W_t$ is such that for all $f \in C^\infty(M)$
$$
f(W_t) - f(W_0) - \frac12\int_0^t \Delta_Mf(W_s)\ud s 
$$
is a local martingale up to the explosion time of $W_t$. 
\end{definition}

Note that a priori there is no guarantee that Riemannian Brownian motion is defined for all times $t > 0$. It turns out that this relies on the geometry of $M$. We make the following definition.

\begin{definition}[Stochastic completeness]
We say that a Riemannian manifold is \emph{stochastically complete} if the explosion time of its Brownian motion is almost surely infinite. 
\end{definition}


The following proposition gives an important sufficient geometric condition for stochastic completeness (see e.g. \cite[Section 4.2]{Hsu02}).

\begin{proposition}
Assume there exists a finite constant $L$ such that $\Ric \geq L$. Then the manifold is stochastically complete.
\end{proposition}

\subsection{Brownian motion as the solution of an SDE}\label{subsection:BMSDE}

We now turn to the approach of defining Riemannian Brownian motion by solving an appropriate SDE. For an introduction to SDEs on manifolds, we refer to appendix \ref{section:SDEmanifold}. In general, it is not possible to obtain Riemannian Brownian motion by solving an SDE on the manifold $M$ itself. The main issue here is that the Laplace-Beltrami operator $\Delta_M$ is not necessarily a sum of squares. Indeed, if $V_1,\ldots,V_l$ are vector fields on $M$, and $W$ is a $l$-dimensional Brownian motion, one can show that the solution of the Stratonovich stochastic differential equation
$$
\dd X_t = V_l\circ \dd W_t^l
$$
is generated by
$$
L = \sum_{i=1}^l V_i^2
$$ 
Consequently, the Laplace-Beltrami operator cannot generate processes of this type. In order to make this idea work, we need to go to the orthonormal frame bundle $OM$, see appendix \ref{section:framebundle}.\\

Let $H_1,\ldots, H_k \in \Gamma(TOM)$ be the fundamental horizontal vector fields and let $B$ be a $k$-dimensional Euclidean Brownian motion. Let $U_t$ be the process in $OM$ given by
\begin{equation}\label{eq:horBM}
\dd U_t = H_i(U_t)\circ \dd B_t^i, \qquad U_0 = u
\end{equation}
where $\pi u = W_0$. Here $\pi: OM \to M$ denotes projection. We have the following proposition (see \cite[Proposition 3.2.1]{Hsu02}):

\begin{proposition}
The process $W_t$ is Riemannian Brownian motion. 
\end{proposition}

%
%
\begin{remark}
Given a Brownian motion $W_t$ on $M$ (defined as the process generated by $\frac12\Delta_M$), there is a unique horizontal semimartingale $U_t$ in $OM$ such that $\pi U_t = W_t$. Furthermore, $U_t$ and $W_t$ have the same explosion times. 
\end{remark}

\subsection{Radial process of Brownian motion}

We conclude this section by studying the radial part of Brownian motion, i.e., the distance of Brownian motion to its starting point. If $W_t$ is Riemannian Brownian motion started at $W_0 = x_0 \in M$, then we define the \emph{radial process}
$$
R_t = d(W_t,x_0).
$$

The following result is Theorem 3.5.1 in \cite{Hsu02}.

\begin{theorem}
Let $W_t$ be a Riemannian Brownian motion started at some $x_0 \in M$. There exists a one-dimensional Euclidean Brownian motion $B_t$ and a nondecreasing process $L_t$, which only increases on the cutlocus of $x_0$, such that
$$
R_t = B_t + \frac12\int_0^t \Delta_Mr(W_s)\ud s - L_t
$$
for all $t$ less than the explosion time of $W$.
\end{theorem}

Our aim is to control probabilities of the type $\PP(T_\delta \leq \tau)$, where $T_\delta$ is the exit time of the geodesic ball $B(x_0,\delta)$ of radius $\delta$ around $x_0$. We have the following proposition, which is an adaptation of Theorem 3.6.1 in \cite{Hsu02}. For a proof we refer to Appendix \ref{section:proofs_propositions_lemmas}. 

\begin{proposition}\label{prop:stopping}
Let $M$ be a complete Riemannian manifold of dimension $k$ and assume there exists an $L \geq 1$ such that the Ricci curvature is bounded from below by $-L$. Fix $p \in M$ and let $W_t$ be standard Riemannian Brownian motion started almost surely in $x_0$. Then for any $\delta > \sqrt{2 kL\tau}$ and any $\tau \geq 0$ it holds that
$$
\PP\left(\sup_{0\leq t \leq \tau} d(W_t,x_0) \geq \delta\right) \leq 2e^{-\frac12\frac{(kL\tau - \frac12\delta^2)^2}{\delta^2\tau}}.
$$ 
\end{proposition}


\section{Random walks on a manifold}\label{section:geodesicRW}

The most basic setting for Mogulskii's theorem is that of random walks on $\bR^k$ with independent and identically distributed increments. We generalize this type of random walks to the setting of manifolds. One way to generate random walks with independent increments is via geodesic random walks introduced in Section \ref{section:geo_RW_eplicit}. Afterwards, in Section \ref{section:RW_identically_distributed_steps}, we generalize the concept of identically distributed increments for geodesic random walks. We conclude this section by giving some examples of random walks with independent and identically distributed increments in Section \ref{section:RW_examples}.

\subsection{Geodesic random walks} \label{section:geo_RW_eplicit}


We start by defining \emph{geodesic random walks} $\{\cS_n\}_{n \geq 0}$ on $M$. We follow  \cite{Jor75}.

\begin{definition} \label{definition:geodesic_random_walk}
	A sequence of random variables $\{\cS_n\}_{n \geq 0}$ on $M$ is called a \emph{geodesic random walks} on $(M,g)$ with \emph{increments} $\{X_{n+1}\}_{n \geq 0}$, $X_n \in T_{\cS_{n}}$ if there is a collection of measures $\{\mu_x\}_{x\in M}$ with $\mu_x \in \cP(T_xM)$ such that: 
	\begin{enumerate}[(a)]
		\item The increments $\{X_n\}_{n \geq 1}$ are independent and $X_{n+1}$ has distribution $\mu_{S_n}$;
		\item The steps are given deterministically as a function of the increments: $\cS_{n+1} := \Exp_{\cS_n}(X_{n+1})$.
	\end{enumerate}
%
%
%
%
\end{definition}

If $M = \bR^k$, the exponential map is given by addition, i.e., $\Exp_x(V) = x + V$, so a geodesic random walk reduces to a random walk with location-dependent step distribution. In this Euclidean setting, we are able to rescale a random walk: $\alpha \cS_n$, $\alpha \in \bR$. On a general manifold this is not possible. The increments, however, can be rescaled by $\alpha$.

\begin{definition}
	Let $\{\cS_n\}_{n \geq 0}$ be a geodesic random walk on $(M,g)$ with increments $\{X_{n+1}\}_{n \geq 0}$,  generated with the collection of measures $\{\mu_x\}_{x\in M}$. Let $\alpha \in \bR$.
	The \emph{rescaled geodesic random walk} $\alpha * S_n$ is the geodesic random walk generated by the collection of measures $\{\mu_{x,\alpha}\}_{x \in M}$ given by
	\begin{equation*}
	\mu_{x,\alpha} = \mu_x \circ m_\alpha^{-1},
	\end{equation*}
	where $m_\alpha : T_x M \rightarrow T_xM$ is given by $m_\alpha(V) = \alpha V$.
\end{definition}

To rephrase the definition of a rescaled geodesic random walk, the construction of $\alpha*\Ss_n$ is completely analogous to that of $\Ss_n$, the only difference being that if $X_{n+1}$ has distribution $\mu_{\Ss_n}$ one should replace for any $n$, the step $\Ss_{n+1} = \Exp_{\Ss_n}(X_{n+1})$ by $\alpha*\Ss_{n+1} = \Exp_{\alpha*S_n}(\alpha X_{n+1})$.

\subsection{Identically distributed increments} \label{section:RW_identically_distributed_steps}

We proceed by introducing an analogue notion of identically distributed increments. As the increments $X_1,X_2,\ldots$ in general do not live in the same tangent space, they are not immediately comparable. However, parallel transport allows us to identify tangent spaces, and thus to compare tangent vectors from different tangent spaces. 

\begin{definition}[Identical distributions on tangent spaces]\label{definition:consistency}
Let $\{\mu_x\}_{x\in M}$ be a collection of measures, such that for each $x\in M$, $\mu_x$ is a probability measure on $T_xM$. We say the distributions are \emph{identical} if the measures satisfy the following \emph{consistency property}:
For all $y,z \in M$ and all $\gamma:[a,b] \to M$ smooth curves with $\gamma(a) = y$ and $\gamma(b) = z$ it must hold that $\mu_z = \mu_y\circ\tau_{\gamma,ab}^{-1}$. 
\end{definition}

The consistency property in the above definition essentially says that the collection of measures $\{\mu_x\}_{x \in M}$ is invariant under parallel transportation.

\begin{remark}
In $\RR^k$, this assumption implies that the measure $\mu_x$ does not depend $x$. Indeed, in $\RR^k$, all tangent spaces can be identified with $\RR^k$ itself, and parallel transport along the straight line between two points is simply the identity. 
\end{remark}

For a probability measure $\mu_x$ on $T_xM$ we define the log moment generating function $\Lambda:T_x^*M \to \RR$ by
$$
\Lambda_x(p) = \log\int_{T_xM} e^{\inp{v}{p}}~\ud\mu_x(v).
$$

The following proposition gives an important equivalent characterization of the consistency property. 

\begin{proposition}\label{prop:logmgf}
Let $\{\mu_x\}_{x\in M}$ be a collection of measures such that $\mu_x$ is a probability measure on $T_xM$ for every $x \in M$. Assume that $\Lambda_x(p) < \infty$ for all $x \in M$ and all $p \in T_x^*M$. The following are equivalent
\begin{enumerate}[(a)]
\item \label{prop:logmgf_consistency} The collection $\{\mu_x\}_{x\in M}$ satisfies the consistency property as in Definition \ref{definition:consistency}.
\item \label{prop:logmgf_laplace_transform} For all $x,y \in M$ and all smooth curves $\gamma:[a,b] \to M$ with $\gamma(a) = x$ and $\gamma(b) = y$ and for all $p \in T_x^*M$ we have
$$
\Lambda_x(p) = \Lambda_y(\tau_{\gamma,ab}p).
$$

\end{enumerate}
\end{proposition}

\begin{proof}
We first prove that \eqref{prop:logmgf_consistency} implies \eqref{prop:logmgf_laplace_transform}. Fix $x,y \in M$ and $\gamma:[a,b] \to M$ a  smooth curve with $\gamma(a) = x$ and $\gamma(b) = y$. Let $p \in T_x^*M$. Writing $\tau_{xy} = \tau_{\gamma,ab}$ we find
\begin{align*}
\Lambda_x(p)
&=
\log\int_{T_xM} e^{\inp{v}{p}} \ud\mu_x(v)
\\&=
\log\int_{T_xM} e^{\inp{\tau_{xy}v}{\tau_{xy}p}} \ud\mu_x(v)
\\&=
\log\int_{T_yM} e^{\inp{w}{\tau_{xy}p}} \ud\mu_y(w)
\\&=
\Lambda_y(\tau_{xy}p).
\end{align*}
Here, the second line follows from the fact that the duality pairing is invariant under parallel transport and the third line follows from the consistency assumption of the collection of measures.

For the reverse implication, fix $x,y \in M$ and let $\gamma:[a,b]\to M$ be a smooth curve with $\gamma(a) = x$ and $\gamma(b) = y$. A similar argument as above shows that the log moment generating function of $\mu_x \circ \tau_{xy}^{-1}$ coincides with the log moment generating function of $\mu_y$. 
Because the moment generating function determines the distribution, we conclude that $\mu_x\circ\tau_{xy}^{-1} = \mu_y$ as desired.
\end{proof}

Finallly, let us consider the differentiability properties of the map $\Lambda: T^*M \to \RR$ given by $\Lambda(x,p) := \Lambda_x(p)$.

\begin{proposition}\label{prop:logmgf_cont_diffb}
Let $\{\mu_x\}_{x\in M}$ be a collection of measures as in Definition \ref{definition:consistency} with associated log moment generating functions $\Lambda_x$. Assume that $\Lambda_x(p) < \infty$ for all $p \in T_x^*M$ and all $x \in M$. Then the map $\Lambda: T^*M \to \RR$ given by $\Lambda(x,p) := \Lambda_x(p)$ is continuously differentiable.
\end{proposition}

\begin{proof}
Fix $x \in M$ and take $U \subset M\setminus C_x$ open such that $x \in U$. For $y \in U$ there exists a unique geodesic of minimal length connecting $x$ and $y$. By Proposition \ref{prop:logmgf} we can write
$$
\Lambda(y,p) = \Lambda(x,\tau_{yx}p)
$$
for all $y \in U$, where $\tau_{xy}$ denotes parallel transport along the geodesic of minimal length between $y$ and $x$. 

As parallel transport is given as the solution of a system of linear differential equations with smooth coefficients, we find that the map $(y,p) \to \tau_{yx}p$ is smooth. Furthermore, as $\Lambda_x$ is finite on $T_x^*M$, it is continuously differentiable (cf. \cite[Lemma 2.2.31]{DZ98}). Because $\Lambda$ is the composition of continuously differentiable maps on $T^*U$, it is continuously differentiable. As this holds for any $x \in M$, the claim follows.  
\end{proof}

\subsection{Examples}\label{section:RW_examples}

We give some examples of collections of measures $\{\mu_x\}_{x\in M}$ satisfying Definition \ref{definition:consistency}

\begin{example}[Uniform distribution on a ball]
Fix $r > 0$. For any $x \in M$, let $\mu_x$ be the uniform distribution on $\left\{V \in T_xM \, \middle| \, \sn{V}_g \leq r\right\} \subseteq T_xM$. To see that this collection of measures satisfies the consistency property, one simply has to observe that parallel transport is an isometry between tangent spaces, thus mapping balls of same radii in different tangent spaces bijectively onto each other.
\end{example}

The next example will be used in a later section to indicate the connection between Mogulskii's theorem and Schilder's theorem.

\begin{example}[Normal distribution]\label{example:normaldistribution}
We now want to consider geodesic random walks with normally distributed increments. For this, we define what we consider to be a standard normal distribution on $T_xM$ and show that it satisfies the consistency property. We say that $V$ has a standard normal distribution if for some basis (equivalently, all bases) $E_1,\ldots,E_N$ of $T_xM$ it holds that
$$
(V^1,V^2,\ldots,V^N) \sim \Nn(0,G^{-1}(x))
$$
where $V = V^iE_i$ and $G(x)$ is the matrix of the metric tensor at $x$ with respect to the basis $E_1,\ldots,E_N$. This is well-defined, because $G^{-1}(x)$ transforms tensorially under coordinate transformations.



To show that this collection of measures satisfies the consistency property in Definition \ref{definition:consistency}, we make use of Proposition \ref{prop:logmgf}. We compute the log moment generating function $\Lambda_x$ of $\mu_x$. For this, we will show that
$$
\inp{V}{p} \sim N(0,|p|^2_{g(x)}).
$$
for any $p \in T_x^*M$. Let $(E_1,\ldots,E_N)$ be a basis for $T_xM$ and $(E_1^*,\ldots,E_N^*)$ the corresponding dual basis of $T_x^*M$. Write $V = V^iE_i$ and $p = p_jE_j^*$. Then
$$
\inp{V}{p} = p_iV^i
$$
This has a normal distribution with mean 0 and variance $p^TG^{-1}(x)p = |p|_{g(x)}^2$.
Using this, the log moment generating function becomes
$$
\Lambda_x(p) = \log\int_{T_xM} e^{\inp{v}{p}}\ud\mu_x(v) = \frac12|p|_{g(x)}^2.
$$
As parallel transport along any smooth curve is an isometry, we find that \eqref{prop:logmgf_laplace_transform} of Proposition \ref{prop:logmgf} is trivially satisfied and consequently, the collection $\{\mu_x\}_{x\in M}$ satisfies the consistency property.
\end{example}

\begin{remark}
The previous example shows that if we have for all $x \in M$ that $\Lambda_x(p) = f(|p|_{g(x)})$ for some function $f$, independent of $x$, then the measures $\{\mu_x\}_{x\in M}$ satisfy the consistency property in Definition \ref{definition:consistency}. This is for example the case if $\mu_x$ conditioned on the norm is uniformly distributed, and the norm is distributed according to a distribution $\nu$ independent of $x$.
\end{remark}

Finally, we will show that if a geodesic random walk has identically distributed increments, it is sufficient to know the probability distribution in a given tangent space. This leads to an equivalent characterization of a geodesic random walk.

\begin{example}[Equivalent characterization of a geodesic random walk]\label{example:alternative}

Suppose we have fixed an initial point $x_0 \in M$ and a measure $\mu$ on $T_{x_0}M$ with the following property: \emph{For every smooth loop $\gamma:[a,b] \to M$ with $\gamma(a) = \gamma(b) = x_0$ it holds that $\mu = \mu \circ \tau_{\gamma,ab}$, i.e., $\mu$ is invariant under parallel transport along any loop.}

Given such a measure $\mu$, we can construct a family of measures $\{\mu_x\}_{x\in M}$ which satisfies Definition \ref{definition:consistency}. Indeed, given $x \in M$, we take a smooth curve $\gamma:[a,b] \to M$ with $\gamma(a) = x_0$ and $\gamma(b) = x$ and define $\mu_x = \mu\circ \tau_{\gamma,ab}$. The assumption on $\mu$ implies that this is well-defined, i.e. independent of the curve $\gamma$, and that the given collection of measures satisfies the consistency property. Indeed, by arguing in a chart around $x$ and $x_0$ respectively, one can make sure to concatenate a smooth curve from $x$ to $x_0$ to the one from $x$ to $x_0$ in a smooth way to create a smooth loop.

Now if $\tilde X_1,\tilde X_2,\ldots$ are $T_{x_0}M$-valued random variables with distribution $\mu$, one can parallelly transport these along the path of the geodesic random walk to obtain all increments $X_1,X_2,\ldots$ of the walk. 

\end{example}


\section{Main results}\label{section:mainresults}


We start by stating Schilder's theorem for Riemannian Brownian motion. 

\begin{theorem}[Schilder's theorem for Riemannian Brownian motion]\label{theorem:Schilder}
Let $(M,g)$ be a complete Riemannian manifold of dimension $k$. Assume that there exists a constant $L \in \bR$ such that $\Ric \geq L$. Let $W_t$ be a Riemannian Brownian motion and assume that $W_0$ satisfies the LDP in $M$ with good rate function $I_0$. Define for every $\epsilon > 0$ the process $W_t^\epsilon: = W_{t\epsilon}$. Then the trajectories of $\{W^\epsilon\}_{\epsilon > 0}$ satisfy in $D(\RR^+;M)$ the LDP with good rate function
$$
I_M(\phi) =
\begin{cases}
 I_0(\phi(0)) + \frac12\int_0^\infty|\dot \phi(t)|_M^2\ud t &\phi \in H^1(\RR^+;M),\\ 
\qquad \infty &\emph{otherwise}.
\end{cases}.
$$
\end{theorem}

We now turn to the generalization of Mogulskii's theorem for time-scaled geodesic random walks on $M$.

Fix an initial point $x_0 \in M$, and let $\{\mu_x\}_{x\in M}$ be a collection of measures satisfying the consistency property as in Definition \ref{definition:consistency}. Let $\frac 1n*\Ss_n$ be the scaled geodesic random walk with independent, identically distributed steps according to the measures $\{\mu_x\}_{x\in M}$ and starting from $x_0 \in M$. Furthermore, we define the processes
$$
Z_n(t) = \frac1n*\Ss_{\lfloor nt\rfloor}, \qquad t \in \bR^+.  
$$

The generalization of Mogulskii's theorem reads as follows. 

\begin{theorem}[Mogulskii's theorem for Riemannian manifolds]\label{theorem:Mogulskii}
Let $M$ be a complete Riemannian manifold of dimension $k$. 
Let $\{\mu_x\}_{x\in M}$ be a collection of measures on the tangent spaces satisfying the consistency property as in Definition \ref{definition:consistency}. Assume that for every $x \in M$ and $p \in T_x^*M$ we have $\Lambda_x(p) < \infty$ 
For every $n \in \NN$, denote by $\nu_n$ the measure of $Z_n(\cdot)$ (as defined above) in $C(\bR^+;M)$. Assume that $\{Z_n(0)\}_{n\geq 1}$ satisfies the LDP in $M$ with good rate function $I_0$. Then the measures $\{\nu_n\}_{n\geq1}$ satisfy in $C(\bR^+;M)$ the large deviation principle with good rate function
\begin{equation}\label{eq:rfmog}
I(\gamma) = 
\begin{cases}
I_0(\gamma(0)) + \int_0^\infty \Lambda_{\gamma(t)}^*(\dot\gamma(t))\ud t	& \gamma \in \cA\cC(\RR^+;M), \\
\infty												&\emph{otherwise}.\\
\end{cases}
\end{equation}
\end{theorem}

 

For completeness, let us explicitly write down the rate function for a specific example.

\begin{example}\label{example:Mogulskii}
Let $\{\mu_x\}_{x\in M}$ be the collection of standard normal distributions as defined in Example \ref{example:normaldistribution}. There it was shown that these measures satisfy the consistency property as in Definition \ref{definition:consistency}. By Mogulskii's theorem, we find that the process $Z_n(t) = \frac1n*\Ss_{\lfloor nt\rfloor}$ satisfies the large deviation principle in $D(\bR^+,M)$ with good rate function
\begin{equation*}
I(\gamma) = 
\begin{cases}
\frac12\int_0^\infty |\dot\gamma(t))|^2_{g(\gamma(t))} \ud t,	& H^1_{x_0}(\RR^+;M), \\
\infty									&\emph{otherwise}.\\
\end{cases}
\end{equation*}
Here we used that $\Lambda_x(p) = \frac12|p|_{g(x)}^2$ and consequently, $\Lambda_x^*(v) = \frac12|v|_{g(x)}^2$.
\end{example}

\begin{remark}
The rate function obtained in the above example coincides with the one found in Schilder's theorem for Riemannian Brownian motion. In the Euclidean case, this is no coincidence, as the increments of Brownian motion are normally distributed and one can deduce Schilder's theorem from Mogulskii's theorem by discretizing. However, in the Riemannian setting it is not clear if a similar approach works, because the increments of Riemannian Brownian motion are no longer normally distributed with the desired parameters due to the curvature. 

What remains true is the result of Varadhan \cite{Var67} on the short-time asymptotics for the heat kernel, stating that
\begin{equation}\label{eq:Varadhan}
\lim_{t\to0} t\log p_M(x,y,t) = -\frac{d^2(x,y)}{2}.
\end{equation}
Note for a general Riemannian manifold, the Riemannian metric does not satisfy assumptions B and C in Section 2 of \cite{Var67}. However, similarly as done in the proof of Lemma 3.1 in \cite{Var67}, one can use \eqref{eq:Varadhan} to obtain the large deviations for the finite dimensional distributions of Brownian paths once Gaussian bounds for the heat kernel for a general (stochastically complete) Riemannian manifold are established (see e.g. \cite{ATW06}). Using Proposition \ref{prop:stopping} (which replaces Lemma 3.2 in \cite{Var67}), one can follow the argument in proving Theorem 3.3 in \cite{Var67} to obtain the large deviations upper bound in Schilder's theorem. For the lower bound, one can exactly mimic the proof of Lemma 3.4 in \cite{Var67}.

\end{remark}

Finally, we present the generalization of Cram\'er's theorem for geodesic random walks in $M$, which is a corollary of Mogulskii's theorem.

\begin{theorem}[Cram\'er's theorem for Riemannian manifolds]\label{theorem:Cramer}
Let $(M,g)$ be a complete Riemannian manifold of dimension $k$.  
Fix $x_0 \in M$ and let $\{\mu_x\}_{x\in M}$ be a collection of measures on the tangent spaces satisfying the consistency property as in Definition \ref{definition:consistency}. Denote by $\frac1n*\Ss_n$ the associated scaled geodesic random walk.  Assume that the log moment generating function $\Lambda_x$ is everywhere finite. Denote by $\nu_n$ the law of $\frac1n*\Ss_n$ in $M$. Then $\{\nu_n\}_{n\geq1}$ satisfies in $M$ the LDP with good rate function
$$
I_M(x) = \min\{\Lambda_{\mu_{x_0}}^*(\dot\gamma(0))|\gamma:[0,1]\to M \emph{ geodesic}, \gamma(0) = x_0, \gamma(1) = x\}.
$$
\end{theorem}


Let us also provide an explicit example of a rate function as in Cram\'er's theorem.

\begin{example}
Continuing Example \ref{example:Mogulskii}, let us also find the corresponding rate function for the end points of the rescaled geodesic random walk with normal increments.  Recall that $\Lambda^*_{\mu_{x_0}}(v) = \frac12|v|_{g(x_0)}$. Now suppose we have some geodesic $\gamma:[0,1]\to M$ with $\gamma(0) = x_0$ and $\gamma(1) = x$. As $\dot \gamma$ is parallel along $\gamma$, one finds
$$
d(x_0,x) \leq \int_0^1 |\dot\gamma(t)|\ud t = |\dot\gamma(0)|_{g(x_0)}.
$$
In particular, there is at least one geodesic $\tilde\gamma$ for which equality holds. Consequently, we find that
\begin{align*}
I_M(x)
&=
\inf\left\{\Lambda_{\mu_{x_0}}^*(\dot\gamma(0)) \, \middle| \,  \gamma:[0,1]\to M \text{ geodesic}, \gamma(0) = x_0, \gamma(1) = x\right\}.
\\
&=
\frac12 d(x_0,x)^2.
\end{align*} 
\end{example}

\section{A proof of Schilder's theorem, Theorem \ref{theorem:Schilder}, via embedding}\label{section:Schilderembedding}

In this section we provide, a new proof of Schilder's theorem for Riemannian Brownian motion on $M$ (Theorem \ref{theorem:Schilder}). We use the orthonormal frame bundle $OM$ and Freidlin-Wentzell theory in Euclidean space by embedding $OM$ into some Euclidean space. For the relevant terminology regarding orthonormal frame bundles, we refer to Appendix \ref{section:framebundle}.\\


{\it Sketch of proof.} We first give a sketch of the proof. 
First note that it suffices to show that for any $T >0$ the LDP holds in $C([0,T];M)$ with good rate function given by
\begin{equation}\label{equation:RFT}
I_M(\phi) =
\begin{cases}
 I_0(\phi(0)) + \frac12\int_0^T|\dot \phi(t)|_M^2\ud t &\phi \in H^1([0,T];M),\\ 
\qquad \infty &\emph{otherwise}.
\end{cases}.
\end{equation}

To show this, observe that if $W_t^\epsilon$ is a rescaled Riemannian Brownian motion, then its horizontal lift $U_t^\epsilon$ to $OM$ is a rescaled horizontal Brownian motion. As explained in Section \ref{subsection:BMSDE}, $U_t^\epsilon$ satisfies
$$
\dd U_t^\epsilon = H_i(U^\epsilon_t)\circ \dd B_t^{\epsilon,i}, \qquad U_0^\epsilon = u_0
$$
where $B_t^\epsilon$ is a rescaled $\RR^k$-valued Brownian motion. Using Whitney's embedding theorem, we can embed $OM$ into a Euclidean space $\RR^N$ and push-forward the SDE, making use of proposition \ref{prop:difSDE} to relate the solutions. This results in an SDE on $\RR^N$ driven by a Euclidean Brownian motion, the solution of which remains inside the embedding of the manifold. 

Using bump functions, we can assure that the diffusion matrix has compact support, and consequently is Lipschitz. This allows us to apply the Freidlin-Wentzell theory in Euclidean space, giving us the LDP in $C([0,T];\RR^N)$. By the contraction principle, this also gives us the LDP for $U_t^\epsilon$ in $C([0,T];OM)$ and consequently also for $W_t^\epsilon$ in $C([0,T];M)$, at least in the case where the diffusion matrix has compact support. An exponential tightness like argument using Proposition \ref{prop:stopping} then gives the LDP also in the noncompact case.\\

{\it Proof of Theorem \ref{theorem:Schilder}.} Let us now provide the further details. First observe that the lower bound on the Ricci curvature implies that $M$ is stochastically complete, i.e., that the explosion time of the Brownian motion $W_t$ is almost surely infinite.

Fix $T > 0$. As explained above, it is sufficient to show that the LDP holds in $C([0,T];M)$ with good rate function given by \eqref{equation:RFT}.\\

First, let $O \in \Gamma(OM)$ be a smooth section of the orthonormal frame bundle. Define the function $F: M \to OM$ by setting $F(x) = O(x)$. Define $U_0^\epsilon := F(X_0^\epsilon)$. As $F$ is continuous, the contraction principle implies that $U_0^\epsilon$ satisfies in $OM$ the LDP with good rate function $I_0^{OM}$ given by
$$
I_0^{OM}(u) = \inf\{I_0(x)| F(x) = u\}
=
\begin{cases}
I_0(x)		& u(x) = O(x)\\
\infty		& \mbox{otherwise}\\
\end{cases}
$$

For every $m \in \NN$, define $K_m := \overline{B(0,m)}$. As $M$ is complete, $\overline{B(0,m)}$ is the image of $\{V \in T_xM| |V| \leq m\}$ under the exponential map. By continuity of the exponential map we conclude that $K_m$ is compact. 

Let $\varphi:M \to \RR$ be a smooth function, $\varphi \equiv 1$ on $K_m$ and with compact support. We extend $\varphi$ to $OM$ by defining it to be constant on fibres. Abusing notation, we call this extension $\varphi$ as well. Consider the process $U_t^{\epsilon,m}$ in $OM$ satsifying 
\begin{equation}\label{eq:SDE1}
\dd U_t^{\epsilon,m} = \varphi H_i(U_t^{\epsilon,m})\circ\dd B_t^{\epsilon,i}, \qquad U_0^{\epsilon,m} \stackrel{D}{=} U_0^\epsilon,
\end{equation}

By Whitney's embedding theorem there exists an $N \in \NN$ and a smooth embedding $\iota: OM \to \RR^N$. We push SDE \eqref{eq:SDE1} forward to $\iota(OM)$ to obtain the SDE
\begin{equation}\label{eq:SDE2}
\dd V_t^{\epsilon,m} = \iota_*(\varphi H_i)(V_t^{\epsilon,m})\circ\dd B_t^{\epsilon,i}, \qquad V_0^{\epsilon,m} = \iota(U_0^\epsilon).
\end{equation}

Because $\varphi$ has compact support in $M$, the continuity of $\iota$ implies that the vector fields $\iota_*(\varphi H_i)$ have compact support, and are hence bounded and Lipschitz continuous.  

As $\iota(OM)$ is a closed submanifold of $\RR^N$, we can extend the vector fields $\iota_*(\varphi H_i)$ to bounded, Lipschitz continuous vector fields on $\RR^N$, which we will denote by $\widetilde{\iota_*(\varphi H_i)}$. This gives us the following SDE on $\RR^N$:
$$
\dd \widetilde V_t^{\epsilon,m} = \widetilde{\iota_*(\varphi H_i)}(\widetilde V_t^{\epsilon,m})\circ\dd B_t^{\epsilon,i}, \qquad \widetilde V_0^{\epsilon,m} = \iota(U_0^\epsilon).
$$
Because $\iota$ is a diffeomorphism and $\iota(OM)$ is closed in $\RR^N$, $\iota(U_0^\epsilon)$ satisfies in $\RR^N$ the LDP with good rate function
$$
I_0^{\RR^N}(v) =
\begin{cases}
I_0^{OM}(\iota^{-1}v)	&	v\in \iota(OM)\\
\infty				&	\mbox{otherwise}\\
\end{cases}
$$
By Theorem \ref{thm:FW}, the trajectories of $\widetilde V_t^{\epsilon,m}$ satisfy the LDP in $C([0,T];\RR^N)$ with good rate function
\begin{equation}
\begin{split}
I_{\RR^N}^m(f) = \inf\Bigg\{I_0^{\RR^N}(f(0)) + \frac12\int_0^T |\dot g(t)|_{\RR^k}^2\ud t \Bigg | & g \in H_1^0([0,T];\RR^k), \\ &\dot f(t) = \dot g^i(t)\widetilde{\iota_*(\varphi H_i}(f(t))\Bigg\}.
\end{split}
\end{equation}

 Oberve that if $f(0) \notin \iota(OM)$, then $I_{\RR^N}(f) = \infty$. If $f(0) \in \iota(OM)$, then  the existence of such a $g$ as in the rate function implies that $f([0,T]) \subseteq \iota(OM)$ because the vector fields $\widetilde{\iota_*(\varphi H_i)}$ are tangent to $\iota(OM)$ at points of $\iota(OM)$. To see this, a similar proof (but adjusted to the deterministic case) as that of \cite[Proposition 1.2.8]{Hsu02} can be used. Hence, $I^m_{\RR^N}$ is infinite outside $C([0,T];\iota(OM))$. As the latter is a closed subset of $C_([0,T];\RR^N)$ (as $\iota(OM)$ is closed in $\RR^N$), we conclude that $V_t^{\epsilon,m}$ satisfies the LDP in $C([0,T];\iota(OM))$, where the rate function $I^m_{\iota(OM)}$ is simply the restriction of $I^m_{\RR^N}$. As the process remains in $\iota(OM)$, we find that the rate function is given by
\begin{equation}
\begin{split}
I^m_{\iota(OM)}(f) = \inf \Bigg\{I_0^{OM}(\iota^{-1}(f(0))) + \frac12\int_0^T |\dot g(t)|_{\RR^k}^2\ud t \Bigg | &g \in H_1^0([0,T];\RR^k), \\& \dot f(t) = \dot g^i(t)\iota_*(\varphi H_i)(f(t))\Bigg\}.
\end{split}
\end{equation}

Now observe that as $\iota$ is an embedding, $\iota(U^{\epsilon,m}_t)$ solves SDE \eqref{eq:SDE2} with initial value $\iota(U_0^\epsilon)$ if and only if $U_t^{\epsilon,m}$ solves SDE \eqref{eq:SDE1} with initial value $U_0^\epsilon$. Consequently, by the contraction principle we get the LDP in $C([0,T];OM)$ for the trajectories of $U_t^{\epsilon,m}$ with good rate function given by
\begin{equation}
\begin{split}
I^m_{OM}(h) =  I_{\iota(OM)}(\iota\circ h) =  
\inf\Bigg\{I_0^{OM}(h(0)) + \frac12\int_0^T |\dot g(t)|_{\RR^k}^2\ud t \Bigg | g \in H_1^0([0,T];\RR^k),\\ \frac{\dd}{\dd t}(\iota\circ h)(t) = \dot g^i(t)\iota_*(\varphi H_i)(\iota\circ h(t))\Bigg\}.
\end{split}
\end{equation}

Now observe that as $\iota$ is a smooth embedding we have $\dot{(\iota\circ h)}(t) = \dot g^i(t)\iota_*(\varphi H_i)(\iota\circ h(t))$ if and only if $\dot h(t) = \dot g^i(t)(\varphi H_i)(h(t))$. 
But then we can rewrite the rate function $I^m_{OM}$ as
\begin{equation}
\begin{split}
I^m_{OM}(h) = \inf\Bigg\{I_0^{OM}(h(0)) + \frac12\int_0^T |\dot g(t)|_{\RR^k}^2\ud t \Bigg | g \in & H_1^0([0,T];\RR^k),  \\ &\dot h(t) = \dot g^i(t)(\varphi H_i)(h(t))\Bigg\}.
\end{split}
\end{equation}

In particular, we see that if $\pi h([0,T]) \subseteq K_m$ then $I^m_{OM}(h)$ can only be finite if $h$ is a horizontal curve in $OM$.  Indeed, if $h$ is not horizontal, the set over which we take the infimum is empty. Notice that for this we use that $\varphi \equiv 1$ on $K_m$. In particular, for every horizontal curve $h$, there exists precisely one $g \in H_0^1([0,T];\RR^k)$ such that $\dot h(t) = \dot g^i(t)H_i(h(t))$ and $h(0) = u_0$, namely the antidevelopment of $h$ along initial frame $u_0$. 

By the continuity of $\pi$, the contraction principle implies the LDP for the trajectories of $W_t^{\epsilon,m} = \pi(U_t^{\epsilon,m})$ in $C([0,T];M)$ with good rate function $I^m_M(f)$ given by
$$
I^m_M(f) = \inf\{I^m_{OM}(\hat f) | \hat f \in C([0,T];OM) \mbox{ with } \pi(\hat f) = f\} .
$$

We now show how to simplify this expression when $f([0,T]) \subseteq K_m$. As discussed above, in this case $I^m_{OM}(\hat f)$ can only be finite if $\hat f$ is horizontal. The infimum must thus be taken over the possible horizontal lifts of $f$. As $I_0^{OM}$ is only finite for the frame $O$, we only to consider the horizontal lift of $f$ via the initial frame $O(f(0))$. But then the rate function can be written as
$$
I^m_M(f) = I_0(f(0)) + \frac12\int_0^T |\dot g(t)|_{\RR^k}^2\ud t, \qquad g \mbox{ the anti-development of } f,
$$
Denoting by $h$ the horizontal lift of $f$, we obtain $\dot g(t) = h^{-1}(t)\dot f(t)$. Using that $h$ is an orthonormal frame, and thus an isometry, we find that
$$
|\dot g(t)|_{\RR^k} = |h^{-1}(t)\dot f(t)|_{\RR^k} = |\dot f(t)|_M.
$$
This shows that, at least for $f$ such that $f([0,T]) \subset K_m$, the rate function is given by
$$
I^m_M(f) = I_0(f(0)) + \frac12\int_0^T|\dot f(t)|_M^2\ud t.
$$

We complete the proof by showing that the trajectories of $W_t^\epsilon$ satisfy the LDP in $C([0,T];M)$ with good rate function
$$
I_M(f) =
\begin{cases}
I_0(f(0)) + \frac12\int_0^T|\dot f(t)|_M^2\ud t, &f \in H^1([0,T];M)\\ 
\qquad \infty &\mbox{otherwise}
\end{cases}.
$$

For this, for every $\epsilon > 0$, denote $T_{K_m}^\epsilon$ the exit time of $W_t^\epsilon$ from $K_m$. Observe that by the definition of $W_t^{\epsilon,m}$ we have that $W_t^\epsilon$ and $W_t^{\epsilon,m}$ agree up to time $T_{K_m}^\epsilon$. 

Let us first prove the upper bound. For this, let $F \subseteq C([0,T];M)$ be closed. Then
\begin{align*}
\PP(W_t^\epsilon \in F) 	
&=
\PP(W_t^\epsilon \in F| T_{K_m}^\epsilon > 1)\PP(T_{K_m}^\epsilon > 1) + \PP(W_t^\epsilon \in F| T_{K_m}^\epsilon \leq 1)\PP(T_{K_m}^\epsilon \leq 1)\\
&\leq  	 
\PP(W_t^\epsilon \in F \wedge T_{K_m}^\epsilon > 1) + \PP(T_{K_m}^\epsilon \leq 1)\\
&=
\PP(W_t^{\epsilon,m} \in F \wedge T_{K_m}^\epsilon > 1) + \PP(T_{K_m}^\epsilon \leq 1)\\
&\leq
\PP(W_t^{\epsilon,m} \in F \cap C([0,T];K_m)) + 2e^{-\frac12\frac{(kL\epsilon - \frac12m^2)^2}{m^2\epsilon}}.
\end{align*}

Here the last inequality follows from Proposition \ref{prop:stopping}. Using the first part of the proof and noticing that $I_M(f) = I_M^m(f)$ when $f([0,T])\subseteq K_m$ proves the large devations upper bound.


It remains to prove the lower bound. Let $G \subseteq C([0,T];M)$ be open. Fix $g \in G$ and take $\delta > 0$ such that $B(g,\delta) \subseteq G$. Note that there exists an $m \in \NN$ such that for all $h \in B(g,\delta)$ it holds that $h([0,T]) \subseteq K_m$. Consequently, we find that
\begin{align*}
\liminf_{\epsilon \to 0} \epsilon\log\PP(W_t^\epsilon \in G) 
&\geq
\liminf_{\epsilon \to 0} \epsilon\log\PP(W_t^\epsilon \in B(g,\delta))\\
&=
\liminf_{\epsilon \to 0} \epsilon\log\PP(W_t^{\epsilon,m} \in B(g,\delta))\\
&\geq
-I_M(g).\\
\end{align*}
Here, the equality follows from the fact that $W_t^{\epsilon}$ and $W_t^{\epsilon,m}$ have the same distribution up to $T_{K_m}^\epsilon$. The last inequality follows from the LDP for $W_t^{\epsilon,m}$ and the fact that $I_M^m(g) = I_M(g)$ for $g$ such that $G([0,T]) \subseteq K_m$. As the above holds for all $g \in G$, this proves the lower bound.\\

Finally, to see that $I_M$ is a good rate function, note that $I_M = \inf_m I_M^m$ and that the $I_M^m$ are good rate functions. 

\begin{remark}
In a similar way as done in the final step of the above proof, one can also show that the LDP holds for $U_t^\epsilon$ and not only for $U_t^{\epsilon,m}$. Indeed, let $T_{K_m}^\epsilon$ be the exit time of $W_t^\epsilon$ from $K_m$. As horizontal lifts are unique, the fact that $W_t^\epsilon$ and $W_t^{\epsilon,m}$ agree up to time $T_{K_m}^\epsilon$ implies that $U_t^\epsilon$ and $U_t^{\epsilon,m}$ agree up to time $T_{K_m}^\epsilon$. 

To prove the upper bound, let $F \subseteq C([0,T];OM)$ be closed. A similar estimate as above shows that
$$
\PP(U_t^\epsilon \in F) \leq \PP(U_t^{\epsilon,m} \in F \cap C([0,T];OK_m)) + \PP(T_{K_m}^\epsilon \leq 1)
$$
Noticing that $I_{OM}(h) = I_{OM}^m(h)$ whenever $h([0,T]) \subseteq OK_m$, a similar argument as above proves that
$$
\limsup_{\epsilon \to 0} \epsilon\log\PP(U_t^\epsilon \in F) \leq -\inf_{h \in F} I_{OM}(h).
$$

For the lower bound, let $G \subseteq C([0,T];OM)$ be open. Fix $g \in G$ and $\delta > 0$ such that $B(g,\delta) \subseteq G$. Note that there exists an $m \in \NN$ such that for all $h \in B(g,\delta)$ it holds that $\pi h([0,T]) \subseteq K_m$, where we possibly have to shrink $\delta$. By a similar argument as in the proof above, we obtain also the lower bound.
\end{remark}


\section{The Hamilton-Jacobi equation and its connection to large deviations} \label{section:abstract_HJ_to_LDP}
 \label{section:viscosity_solutions}

Our proof of Mogulskii's theorem and our second proof of Schilder's theorem will be based on the semigroup and operator convergence arguments introduced by \cite{FK06}. We start by discussing their general strategy in Section \ref{subsection:sketchproofMogulskii}. This will motivate the subsequent sections in which we introduce various techniques and corresponding results.

\subsection{Strategy for proving the large deviation principle}\label{subsection:sketchproofMogulskii}

In our proof of the large deviation principle, we follow the approach introduced by Feng and Kurtz \cite{FK06}. This approach is based on a variant of the projective limit theorem combined with the inverse contraction principle. 

Namely, if a sequence of processes is exponentially tight in the Skorokhod space, then it suffices to establish large deviations of the finite-dimensional distributions. The resulting rate function is given in projective limit form: it is given as the supremum over the rate functions of the finite-dimensional distributions.\\

The large deviation principle for a finite dimensional distribution is established via Bryc's theorem: we prove the convergence of the log-Laplace transforms for a finite dimensional vector of variables. Using the Markov property, this reduces to proving the large deviation principle for time $0$ in addition to proving the convergence of the conditional log-Laplace transforms (arguing for continuous time processes for simplicity):
\begin{equation*}
V_n(t) f(x) := \frac{1}{n} \log \bE\left[e^{nf(X_n(t))} \, \middle| \, X_n(0) = x \right].
\end{equation*}
Writing $S_n(t)$ for the semigroup of conditional expectations corresponding to the Markov process $X_n$ with generator $A_n$, we find that $V_n(t) f = n^{-1} \log S_n(t) e^{nf}$ and that $V_n(t)$ is a semigroup.

Following the theory of weak-convergence of Markov processes, cf. \cite{EK86}, we know that the convergence of linear generators $A_n$ to a limiting linear operator $A$ that generates a semigroup, suffices to establish the convergence of linear semigroups. We follow this approach to prove that there is a limiting non-linear semigroup $V(t)$ of the non-linear semigroups $V_n(t)$.

A formal calculation shows that $H_n$ defined by 
\begin{align*}
\cD(H_n) & := \left\{f \, \middle| \, e^{nf} \in \cD(A_n)\right\}, \\
H_n f & := \frac{1}{n}e^{-nf} A_n e^{nf},
\end{align*}
should be a subset of the (non-linear) generator of the semigroup $V_n(t)$.\\

We therefore aim to show that there is an operator $H \subseteq C_b(M) \times C_b(M)$ satisfying `$H \subseteq \LIM H_n$' and that $H$ generates a semigroup. To do so, we turn to the Crandall-Ligget theorem, \cite{CL71}. We need to verify two conditions:
\begin{itemize}
	\item The maximum principle;
	\item The range condition: for sufficiently many $h \in C_b(M)$ and all $\lambda > 0$ one can find an $f \in \cD(H)$ that solves the Hamilton-Jacobi equation
	\begin{equation} \label{eqn:HJ_basic}
	f - \lambda Hf = h.
	\end{equation}
\end{itemize}
As `$H \subseteq \LIM H_n$', the maximum principle for $H$ is automatic, but for non-linear operators verifying the range condition is a hard, and often impossible, problem. Therefore, we aim to solve \eqref{eqn:HJ_basic} uniquely in the \textit{viscosity sense} and use these viscosity solutions to construct an operator $\hat{H}$ that extends $H$, satisfies `$\hat{H} \subseteq \LIM H_n$' and which satisfies the range condition by construction. The use of viscosity solutions is motivated by the maximum principle. An extension by using viscosity solutions makes sure that the extension $\hat{H}$ also satisfies the maximum principle.

\smallskip

Solving \eqref{eqn:HJ_basic} in the viscosity sense goes via proving existence and uniqueness. First, we consider existence of viscosity sub- and super-solutions. Fix $h \in C_b(M)$ and $\lambda > 0$. Consider the solutions $R_n(\lambda) h$ to the equations $f - \lambda H_n f = h$. Using that `$H \subseteq \LIM H_n$', one can show that
\begin{equation} \label{eqn:semi_relaxed_limits}
\begin{aligned}
\overline{f}(x) & := \sup\left\{\limsup_{n \rightarrow \infty} R_n(\lambda)h(x_n) \, \middle| \, x_n \rightarrow x\right\}, \\
\underline{f}(x) & := \inf\left\{\liminf_{n \rightarrow \infty} R_n(\lambda)h(x_n) \, \middle| \, x_n \rightarrow x\right\},
\end{aligned}
\end{equation}
are a \textit{viscosity sub-} and \textit{viscosity super-solution} to \eqref{eqn:HJ_basic}.  

Existence and uniqueness of viscosity solutions are afterwards established by verifying the \textit{comparison principle}: for all subsolutions $u$ and all supersolutions $v$, we have $u \leq v$. Indeed, note that $\overline{f} \geq \underline{f}$, which, if the comparison principle is satisfied, implies that $f:= \overline{f} = \underline{f}$, implying that $f$ is a viscosity solution to \eqref{eqn:HJ_basic}. In addition, using the comparison principle it is straightforward to check that $f$ must be the unique solution.

\smallskip

Thus, our first aim for the verification of the large deviation principle is two-fold:
\begin{itemize}
	\item Establish that `$H \subseteq \LIM H_n$', see Definition \ref{definition:operator_convergence} below.
	\item Establish the comparison principle for a class of Hamilton-Jacobi equations \eqref{eqn:HJ_basic}, see Section \ref{subsection:abstractcomparison} below.
\end{itemize}

In combination with a verification of exponential tighthness, the large deviation principle established as the consequence of these two steps will yield a rate-function in projective-limit form. To establish the Lagrangian form, we turn to control theory to give a second, explicit construction for viscosity solutions to the Hamilton-Jacobi equation. By the comparison principle, this viscosity solution must equal the solution obtained from our limiting procedure. In turn, this yields an explicit form for the limiting semigroup $V(t)$. This form can afterwards be used to re-express the projective limit form of the rate function in terms of a Lagrangian.\\

In Section \ref{subsection:abstractcomparison}, we introduce the Hamilton-Jacobi equation, discuss viscosity solutions, and a criterion for the uniqueness of viscosity solutions going by the name of the comparison principle. In addition, we give an explicit method to check the comparison principle.

In Section \ref{section:control_theory}, we introduce some basic control theory that we adapt from \cite{FK06} to the setting of manifolds. This does not create any major issues, except for a slight change in notation. 

In Section \ref{section:abstract_LPD}, we connect the Hamilton-Jacobi equation to the large deviation principle. An identification of an explicit form of the solutions to this equation leads to a Lagrangian form of the rate function.

\subsection{Abstract conditions for the comparison principle}\label{subsection:abstractcomparison}

In this section, we give conditions that imply the comparison principle for viscosity sub- and supersolutions to the Hamilton-Jacobi equation 
\begin{equation} \label{eqn:hamilton_jacobi}
f - \lambda Hf = h
\end{equation}
for $h \in C_b(M)$, $\lambda > 0$ and $H \subseteq C_b(M) \times C_b(M)$.

We start by recalling basic definitions.

\begin{definition} \label{definition:viscosity} 
	We say that $u$ is a \textit{(viscosity) subsolution} of equation \eqref{eqn:hamilton_jacobi} if $u$ is bounded, upper semi-continuous and if, for every $f \in \cD(H)$ such that $\sup_x (u(x) - f(x)) < \infty$ and every sequence $x_n \in M$ such that
	\begin{equation*}
	\lim_{n \rightarrow \infty} u(x_n) - f(x_n)  = \sup_x \{u(x) - f(x)\},
	\end{equation*}
	we have
	\begin{equation*}
	\lim_{n \rightarrow \infty} u(x_n) - \lambda Hf(x_n) - h(x_n) \leq 0.
	\end{equation*}
	We say that $v$ is a \textit{(viscosity) supersolution} of equation \eqref{eqn:hamilton_jacobi} if $v$ is bounded, lower semi-continuous and if, for every $f \in \cD(H)$ such that $\inf_x v(x) - f(x) > - \infty$ and every sequence $x_n \in M$ such that
	\begin{equation*}
	\lim_{n \rightarrow \infty} v(x_n) - f(x_n)  = \inf_x \{v(x) - f(x)\},
	\end{equation*}
	we have
	\begin{equation*}
	\lim_{n \rightarrow \infty} v(x_n) - \lambda Hf(x_n) - h(x_n) \geq 0.
	\end{equation*}
	We say that $u$ is a \textit{(viscosity) solution} of Equation \eqref{eqn:hamilton_jacobi} if it is both a sub and a super solution.
\end{definition}

\begin{remark} \label{remark:optimizers_in_visc_sol}
	Consider a viscosity sub-solution $u$ to \eqref{eqn:hamilton_jacobi} and $f \in \cD(H)$ and $x_0$ such that
	\begin{equation*}
	u(x_0) - f(x_0)  = \sup_x \{u(x) - f(x)\}.
	\end{equation*}
	Then it follows that 
	\begin{equation*}
	u(x_0) - \lambda Hf(x_0) - h(x_0) \leq 0.
	\end{equation*}
\end{remark}

\begin{definition} 
	We say that \eqref{eqn:hamilton_jacobi} satisfies the \textit{comparison principle} if for a subsolution $u$ and supersolution $v$ we have $u \leq v$.
\end{definition}

Note that if the comparison principle is satisfied, then a viscosity solution is unique.\\

The main assumption that we will make on our Hamiltonian is that it can be represented as a map on the cotangent bundle.
\begin{assumption} \label{assumption:Hamiltonian_on_cotangentbundle}
	The map $H \subseteq C_b(M) \times C_b(M)$ has a domain $\cD(H)$ such that $C^2_c(M) \subseteq \cD(H)$ and can be represented as 
	\begin{equation*}
	Hf(x) = \cH(x,\dd f(x))
	\end{equation*}
	for a continuous map $\cH : T^*M \rightarrow \bR$ such that for each $x \in M$ the map $p \mapsto \cH(x,p)$ from $T_x^*M$ to $\RR$ is convex.
\end{assumption}


We start with an informal discussion on the verification of the comparison principle in the setting of Assumption \ref{assumption:Hamiltonian_on_cotangentbundle}. Suppose that $u$ is a viscosity subsolution and $v$ a viscosity supersolution to \eqref{eqn:hamilton_jacobi} that in addition satisfy $u,v \in \cD(H)$. Finally, suppose that $x_0$ is such that $u(x_0) - v(x_0) = \sup_x \{u(x) - v(x)\}$. Then, using the viscosity subsolution property of $u$, we find
\begin{equation*}
u(x_0) \leq h(x_0) + \lambda \cH(x_0, \dd v(x_0)).
\end{equation*}
Similarly, using the supersolution property of $v$, we find
\begin{equation*}
v(x_0) \geq h(x_0) + \lambda \cH(x_0, \dd u(x_0)).
\end{equation*}
This yields:
\begin{multline} \label{eqn:informal_comparison_bound}
\sup_x \{u(x) - v(x)\} = u(x_0) - v(x_0) \\
\leq h(x_0) - h(x_0) + \lambda \left[\cH(x_0, \dd v(x_0)) - \cH(x_0, \dd u(x_0)) \right].
\end{multline}

Because $x_0$ is the point where the distance between $u$ and $v$ is maximal, we find $\dd u(x_0) = \dd v(x_0)$ and consequently $u - v \leq 0$.

\smallskip

We used two crucial properties in our informal discussion:

\begin{itemize}
	\item We used that there is a point $x_0$ in which suprema and infima are attained in the difference of two functions, so that we can use Remark \ref{remark:optimizers_in_visc_sol}. This is not always possible. A restriction of the argument to compact subsets of $M$ can be achieved via the use of a containment function, see Definition \ref{definition:goodcontainment}.
	\item That $u,v \in \cD(H)$: i.e. we can use $u$ and $v$ as a test function, both in the same point $x_0$. This is generally not possible: both $u$ and $v$ might not be continuously differentiable. This will be solved by penalizing by a distance function, i.e. we consider two points $(x_\alpha,y_\alpha)$, $\alpha > 0$ such that
	\begin{equation} \label{eqn:optimizers_of_penalization}
	u(x_\alpha) - v(x_\alpha) - \frac{\alpha}{2} d^2(x_\alpha,y_\alpha) = \sup_{x,y \in M}  \left\{u(x) - v(x) - \frac{\alpha}{2} d^2(x,y) \right\}.
	\end{equation}
	For large $\alpha$ the points $x_\alpha$ and $y_\alpha$ are close together, so that $d^2$ behaves like a smooth function and can be used as a test function in the definition of viscosity sub- and supersolutions. Following the argument that leads to \eqref{eqn:informal_comparison_bound}, we end up comparing the Hamiltonian $\cH$ in the points $x_\alpha$ and $y_\alpha$ evaluated in the momenta that are derived from the derivatives of the square of the distance, cf. \eqref{condH:negative:liminf}. We will show that such a comparison leads to a similar bound.
\end{itemize}

We start with a definition of our containment function and two auxiliary results. The first one establishes the existence of optimizers in a perturbed version of \eqref{eqn:optimizers_of_penalization} that takes into account the containment function which allows us to work on compact sets. The second result shows us that there is a smooth function that mimicks the square of the distance if the two points under consideration are close.

\begin{definition}\label{definition:goodcontainment}
	We say that $\Upsilon : M \rightarrow \bR$ is a \textit{good containment function} (for $H$) if
	\begin{enumerate}[($\Upsilon$a)]
		\item $\Upsilon \geq 0$ and there exists a point $x_0$ such that $\Upsilon(x_0) = 0$,
		\item $\Upsilon$ is twice continuously differentiable,  
		\item for every $c \geq 0$, the set $\{x \in M \, | \, \Upsilon(x) \leq c\}$ is compact,
		\item we have $\sup_z \cH(z,\dd \Upsilon(z)) < \infty$.
	\end{enumerate}
\end{definition}

To use the definition of viscosity sub and super-solutions, we use a containment function to restrict our analysis to compact sets. Next, to bound $\sup_x u(x) - v(x)$, we double the number of variables, but penalize having a large distance between both coordinates. The following result is a variant of Lemma 9.2 in \cite{FK06}, Proposition 3.7 in \cite{CIL92} and Lemma A.10 in \cite{CK17}.

\begin{lemma}\label{lemma:doubling_lemma}
	Let $u$ be bounded and upper semi-continuous and $v$ be bounded and lower semi-continuous. Assume that $\Psi : M^2 \rightarrow \bR^+$ is lower semi-continuous and such that $x = y$ if and only if $\Psi(x,y) = 0$. Finally, let $\Upsilon$ be a good containment function.

	Fix $\varepsilon > 0$. For every $\alpha >0$ there exist points $x_{\alpha,\varepsilon},y_{\alpha,\varepsilon} \in E$, such that
	\begin{multline*}
	\frac{u(x_{\alpha,\varepsilon})}{1-\varepsilon} - \frac{v(y_{\alpha,\varepsilon})}{1+\varepsilon} - \alpha \Psi(x_{\alpha,\varepsilon},y_{\alpha,\varepsilon}) - \frac{\varepsilon}{1-\varepsilon}\Upsilon(x_{\alpha,\varepsilon}) -\frac{\varepsilon}{1+\varepsilon}\Upsilon(y_{\alpha,\varepsilon}) \\
	= \sup_{x,y \in M} \left\{\frac{u(x)}{1-\varepsilon} - \frac{v(y)}{1+\varepsilon} -  \alpha\Psi(x,y)  - \frac{\varepsilon}{1-\varepsilon}\Upsilon(x) - \frac{\varepsilon}{1+\varepsilon}\Upsilon(y)\right\}.
	\end{multline*}
	Additionally, for every $\varepsilon > 0$ we have that
	\begin{enumerate}[(a)]
		\item The set $\{x_{\alpha,\varepsilon}, y_{\alpha,\varepsilon} \, | \,  \alpha > 0\}$ is contained in a compact set that equals the closure of its interior
		\item All limit points of $\{(x_{\alpha,\varepsilon},y_{\alpha,\varepsilon})\}_{\alpha > 0}$ are of the form $(z,z)$ and for these limit points we have $u(z) - v(z) = \sup_{x \in M} \left\{u(x) - v(x) \right\}$.
		\item We have 
		\[
		\lim_{\alpha \rightarrow \infty}  \alpha \Psi(x_{\alpha,\varepsilon},y_{\alpha,\varepsilon}) = 0.
		\]
	\end{enumerate}
\end{lemma}

For the function $\Psi$, we would like to use the distance function $\Psi(x,y) := d^2(x,y)$. The distance $d$, however, is not smooth.

\begin{lemma}\label{lemma:injectivity_radius}
	Consider a compact set $K \subseteq M$. Then there is $\delta = \delta_K > 0$ and a smooth function $\Psi = \Psi_K$ with $\Psi(x,y) \in C^\infty(M^2)$ such that $\Psi(x,y) = \frac{1}{2}d^2(x,y)$ if $d(x,y) < \delta$ and $x,y \in K$
\end{lemma}

\begin{proof}
	As $K$ is compact, the continuity of the injectivity radius $i$ implies there exists $\epsilon > 0$ such that $i(K) > \epsilon$. Pick $\delta = \frac13\epsilon$ and let $\Phi:\RR \to \RR$ be a smooth increasing function such that $\Phi(x) = x$ if $|x| < \delta$ and $\Phi(x) = 1$ if $|x| > 2\delta$.  Then the function $\Psi(x,y) = \frac12\Phi(d^2(x,y))$ is as desired. 
\end{proof}

We end this section with the appropriate generalization of the comparison of Hamiltonians that was used in \eqref{eqn:informal_comparison_bound} in the setting of a penalization with the square of the distance. The result is an adaptation, using containment functions, of Lemma 9.3 in \cite{FK06}. A second proof of this result using containment functions, analogous to the setting of this paper, can be found following Proposition A.11 of \cite{CK17}. 

Note that by Lemma \ref{lemma:doubling_lemma} and \ref{lemma:injectivity_radius}, we find that for each fixed $\varepsilon$ the sequence $(x_{\alpha,\varepsilon},y_{\alpha,\varepsilon})$ is contained in a compact set $K$. By Lemma's \ref{lemma:injectivity_radius} and \ref{lemma:doubling_lemma} (c) it thus follows that for large $\alpha$, we can replace $d^2(x_{\alpha,\varepsilon},y_{\alpha,\varepsilon})$ and its exterior derivatives by $\Psi(x_{\alpha,\varepsilon},y_{\alpha,\varepsilon})$ and its exterior derivatives respectively.

\begin{proposition} \label{proposition:comparison_conditions_on_H}
	Let $H$ be an operator satisfying Assumption \ref{assumption:Hamiltonian_on_cotangentbundle}. Fix $\lambda >0$, $h \in C_b(M)$ and consider $u$ and $v$ sub- and super-solution to $f - \lambda Hf = h$. Let $\Upsilon$ be a good containment function. Moreover, for every $\alpha,\varepsilon >0$ let $x_{\alpha,\varepsilon},y_{\alpha,\varepsilon} \in M$ be such that
	\begin{multline} \label{eqn:comparison_principle_proof_choice_of_sequences}
	\frac{u(x_{\alpha,\varepsilon})}{1-\varepsilon} - \frac{v(y_{\alpha,\varepsilon})}{1+\varepsilon} -  \frac{\alpha}{2} d^2(x_{\alpha,\varepsilon},y_{\alpha,\varepsilon}) - \frac{\varepsilon}{1-\varepsilon}\Upsilon(x_{\alpha,\varepsilon}) -\frac{\varepsilon}{1+\varepsilon}\Upsilon(y_{\alpha,\varepsilon}) \\
	= \sup_{x,y \in M} \left\{\frac{u(x)}{1-\varepsilon} - \frac{v(y)}{1+\varepsilon} - \frac{\alpha}{2} d^2(x,y)  - \frac{\varepsilon}{1-\varepsilon}\Upsilon(x) - \frac{\varepsilon}{1+\varepsilon}\Upsilon(y)\right\}.
	\end{multline}

	Suppose that
	\begin{multline}\label{condH:negative:liminf}
	\liminf_{\varepsilon \rightarrow 0} \liminf_{\alpha \rightarrow \infty} \cH\left(x_{\alpha,\varepsilon},\frac{\alpha}{2} \dd d^2(\cdot,y_{\alpha,\varepsilon})(x_{\alpha,\varepsilon})\right) \\
	- \cH\left(y_{\alpha,\varepsilon}, - \frac{\alpha}{2} \dd d^2 (x_{\alpha,\varepsilon},\cdot)(y_{\alpha,\varepsilon})\right) \leq 0,
	\end{multline}
	then $u \leq v$. In other words: $f - \lambda H f = h$ satisfies the comparison principle. 
\end{proposition}

\subsection{Control theory} \label{section:control_theory}

Next, we introduce some basic definitions from control theory, which can be used to write down a viscosity solution for the Hamilton-Jacobi equation. Given the comparison principle, this identifies the resolvent using an explicit formula  featuring the \textit{Lagrangian} $\cL : TM \rightarrow [0,\infty]$, defined as $\cL(x,v) = \sup_{p \in T_x^* M} \{\ip{p}{v} - \cH(x,p)$\}. This Lagrangian keeps track of the cost along a trajectory that will play a central role in the form of the rate function of the large deviation principle. We define a variational semigroup $\bfV(t), t \geq 0$ and resolvent $\bfR(\lambda), \lambda > 0$ in terms of $\cL$:

\begin{equation*}
\bfV(t)f(x) := \sup_{\substack{\gamma \in \cA\cC, \\ \gamma(0) = x}} \left\{f(\gamma(t)) - \int_0^t \cL(\gamma(s),\dot{\gamma}(s)) \dd s\right\},
\end{equation*}
and
\begin{multline*}
\bfR(\lambda) f(x) := \\
\limsup_{t \rightarrow \infty} \sup_{\substack{\gamma \in \cA\cC \\ \gamma(0) = x}} \left\{ \int_0^t \lambda^{-1} e^{-\lambda^{-1}t} \left(f(\gamma(t)) - \int_0^s \cL(\gamma(r),\dot{\gamma}(r) \dd r \right) \dd s \right\}.
\end{multline*}

	The following two results establish the conditions that are needed for the application of the control theory component of \cite[Theorem 8.27]{FK06}.

	The first result, Proposition \ref{proposition:FK8.9}, can be used to establish the path-space compactness of the set of trajectories that start in a compact set and have uniformly bounded Lagrangian cost. The compactness of this set can be used to establish various properties of $\bfV$ and $\bfR$. 
	
	If one additionally assumes that there exists a trajectory with zero cost, which will follow from the much stronger second result, one can infer that, 
	\begin{itemize}
		\item the resolvents approximates the  semigroups as in the Crandall-Ligget theorem, see cf. \cite[Lemma 8.18]{FK06},
		\begin{equation*}
		\lim_{n \rightarrow \infty} \bfR(n^{-1})^{\lfloor nt \rfloor} f(x) = \bfV(t)f(x);
		\end{equation*}
		\item the resolvent $\bfR(\lambda) h$ is a viscosity sub-solution to $f - \lambda H f = h$.
	\end{itemize}
		
	The second result, Proposition \ref{proposition:FK8.11}, is crucial in establishing that the lower semi-continuous regularization of $\bfR(\lambda)h$ is a viscosity supersolution to the Hamilton-Jacobi equation $f - \lambda H f = h$. 

	\smallskip
	
	Thus, if the comparison principle holds for $f - \lambda H f = h$, $\lambda > 0$ and $h \in C_b(M)$, the variational resolvent gives the unique viscosity solution to the Hamilton-Jacobi equation. This, in turn, means that the variational semigroup is the semigroup generated by the Hamiltonian $H$.


\begin{proposition} \label{proposition:FK8.9}
	Suppose that $\cH : T^*M \rightarrow \bR$ is once continuously differentiable and define $\cL$ as its Legendre transform.
	
	Suppose that there is a good containment function $\Upsilon$ for $\cH$. Then 
	
		\begin{enumerate}[(1)]
		\item $\cL : TM \rightarrow [0,\infty]$ is lower semi-continuous and for each compact set $K \subseteq E$ and $c \in \bR$ the set
		\begin{equation*}
		\left\{(x,v) \in TM \, \middle| \, x \in K, \, \cL(x,v) \leq c \right\}
		\end{equation*}
		is compact in $TM$.
		\item For each compact $K \subseteq M$, $T > 0$ and $0 \leq C < \infty$, there exists a compact set $K' = K'(K,T,C) \subseteq M$ such that $\gamma \in \cA\cC$ and $\gamma(0) \in K$ and
		\begin{equation*}
		\int_0^T \cL(\gamma(s),\dot{\gamma}(s)) \, \dd s \leq C
		\end{equation*}
		implies $\gamma(t) \in K'$ for all $0 \leq t \leq T$.
		\item For each $f \in C^\infty_c(M)$ and compact $K \subseteq M$, there exists a right-continuous non-decreasing function $\psi_{f,K} : \bR^+ \rightarrow \bR^+$ such that $\lim_{r \rightarrow \infty} r^{-1} \psi_{f,K}(r)  = 0$ and
		\begin{equation*}
		|\ip{v}{\dd f(x)}| \leq \psi_{f,K}(\cL(x,v)), \qquad \forall (x,v) \in TM, x \in K. 
		\end{equation*} 
	\end{enumerate}
\end{proposition}

The result can be proven as in Lemma 2 of \cite{Kra16}.

\begin{proposition} \label{proposition:FK8.11}
	Suppose that $\cH : T^*M \rightarrow \bR$ is once continuously differentiable and define $\cL$ as its Legendre transform.
	
	Suppose that there is a good containment function $\Upsilon$ for $\cH$. Then for each $x_0 \in M$ and $f \in \cD(H)$, there exists $\gamma \in \cA \cC$ with $\gamma(0) = x_0$ and for all $t_1 < t_2$:
	\begin{equation*}
	\int_{t_1}^{t_2} H f(\gamma(s)) \dd s \leq \int_{t_1}^{t_2} \ip{\dot{\gamma}(s)}{\dd f(\gamma(s))} - \cL(\gamma(s),\dot{\gamma}(s))) \,  \dd s.
	\end{equation*}
\end{proposition}

\begin{proof}
	The proof can be carried out as in the proof of Lemma 3 of \cite{Kra16}. We prefer to spell this out as it is slightly more involved in our manifold setting. Consider $f \in C_c^\infty(M)$ and $x_0 \in M$. We construct $\gamma \in \cA \cC$ with $\gamma(0) = x_0$ and such that for all $t_1 < t_2$
	\begin{equation}\label{equation:optimizing_flow_in_Young_inequality}
	\int_{t_1}^{t_2} Hf(\gamma(s)) \dd s = \int_{t_1}^{t_2} \ip{\dot{\gamma}(s)}{\dd f(\gamma(s))} - \cL(\gamma(s),\dot{\gamma}(s)) \dd s.
	\end{equation}

	To construct such a $\gamma$ we follow the approach in the proof of lemma 10.21 in \cite{FK06} or lemma 3.4 in \cite{Kra16}. For every $x \in M$, denote by $\Nabla_p\cH(x,p) \in T_p(T^*_xM)$ the derivative of $\cH(x,\cdot)$ with respect to the second variable. As $T_x^*M$ is a vector space, we have $T_p(T^*_xM) \eqsim T^*_xM \eqsim T_xM$, where $\eqsim$ means the spaces are isomorphic as vector spaces. Here, the last identification holds via the Riemannian metric.  Consequently, given $f \in C^1(M)$ we can define the continuous vector field $Q_f$ by
	\begin{equation} \label{eqn:optimizing_vector_field}
	Q_f(x) = \Nabla_p\cH(x,\dd f(x)).
	\end{equation}
	Here, continuity follows from our assumption on $\cH$. From the theory of convex analysis (see e.g. \cite[Section 26]{Roc70}) we find that
	\begin{equation}\label{equation:optimalduality}
	Hf(x) = \cH(x,\dd f(x)) = \inp{Q_f(x)}{\dd f(x)} - \cL(x,Q_f(x)).
	\end{equation}
	
	Thus, if a solution $\gamma$ to \eqref{eqn:optimizing_vector_field} exists, then integration of \eqref{equation:optimalduality} along this $\gamma$ yields \eqref{equation:optimizing_flow_in_Young_inequality}.
	
	We will construct the solution by pasting together local solutions, which exist due to the possibility of arguing via coordinate charts. The size of the interval on which we can guarantee existence of the local solution depends on the size of the vector field. Therefore, if we can give a-priori control on the range of possible solutions, i.e., find a compact set in which a solution is contained, we can bound the size of the vector field. From this we  obtain a uniform lower bound on the length of the interval on which a local solution exists. This allows us to establish existence for the full interval $[t_1,t_2]$.
	
	So suppose $\gamma_f$ is a solution to \eqref{eqn:optimizing_vector_field}. \eqref{equation:optimalduality} implies
	\begin{equation*}
	\int_{t_1}^{t_2} \cL(\gamma_f(s),\dot{\gamma}_f(s)) \dd s \leq 	\int_{t_1}^{t_2} \ip{\dot{\gamma}_f(s)}{\dd f(\gamma_f(s))}  - Hf(\gamma_f(s)) \,\dd s.
	\end{equation*}
	Note that as $f$ is constant outside of a compact set, the map $x \mapsto \cH(x,\dd f(x))$ is bounded from below as $\cH$ is continuous, and $Q_f(x)$ is bounded on compact sets, there is some $C \geq 0$ such that
	\begin{equation*}
	\int_{t_1}^{t_2} \cL(\gamma_f(s),\dot{\gamma}_f(s)) \dd s \leq 	C.
	\end{equation*}
	Using Proposition \ref{proposition:FK8.9}(2), the trajectory $\gamma_f$ remains in some compact set $K' \subseteq M$ for $t \in [t_1,t_2]$. \\
	
	This implies we have a-priori control on the range of a solution to \eqref{eqn:optimizing_vector_field}. On this set, we can bound the size of the vector field and thus give a lower bound for the size of the interval on which we construct a local solution. These local solutions can be patched together to construct a global solution on $[t_1,t_2]$.

\end{proof}

\subsection{Compact containment and the large deviation principle} \label{section:abstract_LPD}

To connect the Hamilton-Jacobi equation to the large deviation principle, we introduce some additional concepts. We consider the following notion of operator convergence. 

\begin{definition} \label{definition:operator_convergence}
	Suppose that for each $n$ we have an operator $(B_n,\cD(B_n))$, $B_n : \cD(B_n) \subseteq C_b(M) \rightarrow C_b(M)$. The \textit{extended limit} $ex-\LIM_n B_n$ is defined by the collection $(f,g) \in C_b(M) \times C_b(M)$ such that there exist $f_n \in \cD(B_n)$ satisfying
	\begin{equation} \label{eqn:convergence_condition}
	\lim_{n \rightarrow \infty} \sup_{x \in K} \left|f_n(x) - f(x)\right| + \left|B_n f_n(x) - g(x)\right| = 0.
	\end{equation}
	For an operator $(B,\cD(B))$, we write $B \subseteq ex-\LIM_n B_n$ if the graph $\{(f,Bf) \, | \, f \in \cD(B) \}$ of $B$ is a subset of $ex-\LIM_n B_n$.
\end{definition}

\begin{assumption} \label{assumption:LDP_assumption}
	 Let $\{r_n\}_{n \geq 1}$, $r_n > 0$, be some sequence of speeds with $\lim_{n \rightarrow \infty} r_n = \infty$.
	\begin{description}
		\item[Continuous time case] Assume that for each $n \geq 1$, we have a linear operator $A_n \subseteq C_b(M) \times C_b(M)$ and existence and uniqueness holds for the $D(\bR^+,M)$ martingale problem for $(A_n,\mu)$ for each initial distribution $\mu \in \cP(M)$. Letting $\bR_{y}^n \in \cP(D(\bR^+,M))$ be the solution to $(A_n,\delta_y)$, the mapping $y \mapsto \bR_y^n$ is measurable for the weak topology on $\cP(D(\bR^+,M))$. Let $X_n$ be the solution to the martingale problem for $A_n$ and set
		\begin{equation*}
		H_n f = \frac{1}{r_n} e^{-r_nf}A_n e^{r_nf} \qquad e^{r_nf} \in \cD(A_n). 
		\end{equation*}
		\item[Discrete time case]  Assume for each $n \geq 1$, we have a transition operator $T_n : C_b(M) \rightarrow C_b(M)$ for a Markov chain. In addition, let $\varepsilon_n >0$ be a sequence of step-sizes with $\varepsilon_n \rightarrow 0$. For each $n$, let $X_n$ be a discrete-time Markov chain with transition operator $T_n$ and time-step $\varepsilon_n$:
		\begin{equation*}
		\bE\left[f(X_n(t)) \, \middle| \, X_n(0) = x \right] = T_n^{\lfloor t \varepsilon_n^{-1} \rfloor} f(x).
		\end{equation*}
		Set
		\begin{equation*}
		H_n f = \frac{1}{r_n \varepsilon_n} \log e^{-r_n f} T_n e^{r_n f}.
		\end{equation*}
	\end{description}
	Suppose that we have an operator $H : \cD(H) \subseteq C_b(M) \rightarrow C_b(M)$ with $\cD(H) = C^\infty_c(M)$ satisfying Assumption \ref{assumption:Hamiltonian_on_cotangentbundle} which satisfies $H \subseteq ex-\LIM H_n$. Finally, assume that the map $\cH : T^*M \rightarrow \bR$ is continuously differentiable.
\end{assumption}

The following result follows along the lines of Proposition A.15 in \cite{CK17}, whose proof is based on Lemma 4.22 in \cite{FK06}. In both references, the result was only proven for the continuous time case, but it can be generalized without problem to the discrete time case.

\begin{proposition} \label{proposition:exp_compact_containment}
	Suppose Assumption \ref{assumption:LDP_assumption} is satisfied and assume that $\Upsilon$ is a good containment function for $H$.
	
	Then the sequence $\{X_n\}_{n \geq 1}$ satisfies the exponential compact containment condition with speed $\{r_n\}_{n \geq 1}$: for every $T > 0$ and $a \geq 0$, there exists a compact set $K_{a,T} \subseteq M$ such that
	\begin{equation*}
	\limsup_{n \rightarrow \infty} \frac{1}{r_n} \log \bP\left[X_n(t) \notin K_{a,T} \text{ for some } t \leq T \right] \leq -a.
	\end{equation*}
\end{proposition}

\begin{theorem} \label{theorem:abstract_LDP}
		Consider the setting of Assumption \ref{assumption:LDP_assumption}. Suppose that $X_n(0)$ satisfies a large deviation principle with speed $\{r_n\}_{n \geq 1}$ and good rate function $I_0$. 
	\begin{enumerate}[(a)]
		\item Suppose that $\Upsilon$ is a good containment function for $H$. Then the processes $\{X_n\}_{n \geq 1}$ are exponentially tight with speed $r_n$ in $D(\bR^+,M)$.
		\item In addition to the assumption in (a), suppose that for each $\lambda > 0$ and $h \in C_b(M)$ the comparison principle is satisfied for  $f - \lambda H f = h$. Then the large deviation principle is satisfied with speed $r_n$ for the processes $X_n$ with good rate function $I$ given by
			\begin{equation*}
			I(\gamma) = \begin{cases}
			I_0(\gamma(0)) + \int_0^\infty \cL(\gamma(s),\dot{\gamma}(s)) \, \dd s   & \text{if } \gamma \in \cA \cC, \\
			\infty & \text{otherwise},
			\end{cases}
			\end{equation*}
			where $\cL : TM \rightarrow [0,\infty]$ is the Legendre transform of $\cH$ given by 
			\begin{equation*}
			\cL(x,v) = \sup_{p \in T^*_xM} \left\{ \ip{v}{p} - \cH(x,p)\right\}.
			\end{equation*}
	\end{enumerate}	
\end{theorem}

\begin{proof}[Proof of Theorem \ref{theorem:abstract_LDP}]
	
	(a) follows from Proposition \ref{proposition:exp_compact_containment} and Corollary 4.19 in in \cite{FK06}. The conditions of Corollary 4.19 can be verified by taking $F = C_c^\infty(M)$.
	
	(b) follows from Theorem 8.27 and Corollary 8.28 in \cite{FK06} by taking $\bfH_\dagger = \bfH_\ddagger =H$ and $\cA f(x,v) = \ip{v}{\dd f(x)}$. Note that definitions and results in the control theory chapter, Chapter 8, of \cite{FK06}, carry over verbatim by replacing product of the state-space and a control space by the tangent space. The conditions for the application of these results have been verified in Propositions \ref{proposition:FK8.9} and \ref{proposition:FK8.11}. Finally, note that the rate function in \cite{FK06} still involves an infimum over control measures. As our Lagrangian is convex in the speed variable, Jensen's inequality gives the final form.
\end{proof}

%
%
%
%


\section{Classical LDP theorems on Riemannian manifolds via the Feng-Kurtz formalism}\label{section:resultsFK}

In this section we prove Theorems \ref{theorem:Schilder}, \ref{theorem:Mogulskii} and \ref{theorem:Cramer} in Sections \ref{subsection:Schilder}, \ref{subsection:Mogulskiiproof} and \ref{subsection:Cramer} respectively. Before doing so we construct a good containment function for the first two theorems in Section \ref{subsection:containment},

\subsection{Good containment function}\label{subsection:containment}
	
	For the proofs of Schilders's and Mogulskii's theorem, cf. Theorems \ref{subsection:Schilder} and \ref{subsection:Mogulskiiproof}, we argue via Theorem \ref{theorem:abstract_LDP} for which we need a good containment function. We construct one containment function that will suffice for both proofs. We use the following proposition

\begin{proposition}\label{proposition:distancelikefunction}
	Let $(M,g)$ be a complete Riemannian manifold of dimension $k$. Fix $x_0 \in M$ and define $r(x) := d(x,x_0)$. There exists a smooth function $f \in C^\infty(M)$ such that $\vn{f -r } \leq 1$ and $|\dd f| \leq 2$.
\end{proposition}

	Consider the function $f$ as in the above proposition and set
	\begin{equation*}
	\Upsilon(x) = \log(1 + f^2(x)).
	\end{equation*}

\begin{lemma}\label{lemma:goodcontainment}
	Let either $\cH$ be given by
	\begin{equation}\label{equation:HamiltonianMogulskii}
	\cH(x,p) = \log \int_{T_xM} e^{\ip{v}{p}}\mu_x(\dd v)
	\end{equation}
	or
	\begin{equation}\label{equation:HamiltonianSchilder}
	\cH(x,p) = \frac12|p|^2_{g(x)}. 
	\end{equation}
	Then $\Upsilon$ is a good containment function for $\cH$.
\end{lemma}

\begin{proof}
	Clearly $\Upsilon \geq 0$ and $\Upsilon(x_0) = 0$, and $\Upsilon \in C^\infty(M)$.
	
	Now fix $c \geq 0$. By the continuity of $\Upsilon$, the set $\{x \in M\, | \, \Upsilon(x) \leq c\}$ is closed.  By definition the set is bounded, and as $M$ is a finite dimensional manifold, also compact. 
	
	Now consider the Hamiltonian $\cH$ in \eqref{equation:HamiltonianMogulskii}. Note that for all $x \in M\setminus\{x_0\}$
	$$
	\dd\Upsilon(x) = \frac{2f(x)}{1 + f^2(x)}\dd f(x).
	$$
	Consequently, $|\dd\Upsilon(x)| \leq |\dd f(x)| \leq C$ because $\dd f$ is uniformly bounded. But then
	$$
	\cH(x,\dd\Upsilon(x)) = \log \int_{T_xM} e^{\ip{v}{\dd \Upsilon(x)}}\mu_x(\dd v) \leq \log \int_{T_xM} e^{C|v|_{g(x)}}\mu_x(\dd v) =: C_x < \infty,
	$$
	where $C_x$ is finite as we assume the log moment generating function of $\mu_x$ to be finite. By the consistency property (as in Definition \ref{definition:consistency}), $C_x$ actually does not depend on $x$. Consequently, we find that $\sup_{x\in M} \cH(x,\dd\Upsilon(x)) < \infty$. That the same holds for the Hamiltonian as in \eqref{equation:HamiltonianSchilder} follows immediately from the uniform boundedness of $\dd\Upsilon$.
\end{proof}

\subsection{Proof of Schilder's Theorem, Theorem \ref{theorem:Schilder}}\label{subsection:Schilder}

In this section we provide an alternative proof of Schilder's theorem for Riemannian Brownian motion based on Theorem \ref{theorem:abstract_LDP}.

\begin{proof}[Proof of Theorem \ref{theorem:Schilder}]
	We verify the conditions for Theorem \ref{theorem:abstract_LDP}. 
	
	\textit{Step 1:} We calculate $H_n$ and limit $H$. The process $W_t$ solves the martingale problem for the operator $\frac12\Delta_M$ and consequently, $W_t^{n^{-1}}$ is generated by $\frac{1}{2n}\Delta_M$. For $f \in C_c^\infty(M)$, we find
	\begin{align*}
	H_nf 		&= \frac1n e^{-nf}\frac{1}{2n}\Delta_Me^{nf}\\
	&= \frac1n e^{-nf}e^{nf}\frac12(\Delta_Mf + n\sn{\dd f}_{g(x)}^2)\\
	&= \frac1{2n}\Delta_Mf + \frac12\sn{\dd f}_{g(x)}^2.
	\end{align*}
	Let $H \subseteq C_b(M) \times C_b(M)$ be given by $\cD(H) = C_c^\infty(M)$ and for $f \in C_c^\infty(M)$: $Hf = \tfrac{1}{2} \sn{\dd f}^2_{g(x)}$ .
	
	It follows that for all $f \in C^\infty_c(M)$, we have
	\begin{equation*}
	\lim_{n \rightarrow \infty} \vn{H_n f - Hf} = 0,
	\end{equation*}
	implying that $H \subseteq \LIM H_n$. Note that $Hf(x) = \cH(x, \dd f(x))$ for $\cH : TM \rightarrow \bR$ of the form $H(x,p) = \tfrac{1}{2} |p|^2_{g(x)}$.
	
	\textit{Step 2:} By Lemma \ref{lemma:goodcontainment} we have a good containment function $\Upsilon$.

	\textit{Step 3:} Fix $\lambda > 0$ and $h \in C_b(M)$. We verify the comparison principle for $f - \lambda Hf = h$ by the application of Proposition \ref{proposition:comparison_conditions_on_H}.
	Let $x_{\alpha,\epsilon},y_{\alpha,\epsilon}$ be as in Proposition \ref{proposition:comparison_conditions_on_H}. We establish \eqref{condH:negative:liminf}. 
	
	Fix $\varepsilon > 0$. By Lemma \ref{lemma:doubling_lemma}, there is a compact $K^\epsilon \subseteq M$ such that  $\{x_{\alpha,\epsilon},y_{\alpha,\epsilon}\, | \, \alpha > 0\}$ is contained in $K^\epsilon$. By the continuity of the injectivity radius and the compactness of $K^\epsilon$, we can find a $\delta > 0$ such that $i(K^\epsilon) \geq \delta > 0$. Then there exists a unique geodesic of minimal length connecting $x$ and $y$. By Proposition \ref{proposition:transport_derivative_distance} we have
	$$
	\dd d^2(\cdot,y)(x) = -\tau_{xy}\dd d^2(x,\cdot)(y),
	$$
	where $\tau_{xy}$ denotes parallel transport along the unique geodesic of minimal length connecting $x$ and $y$. As parallel transport is an isometry, we find
	\begin{align*}
	\Hh\left(x, \frac{\alpha}{2} (\dd d^2(\cdot, y))(x) \right)
	&=
	\frac12 \sn{\frac{\alpha}{2} (\dd d^2(\cdot, y))(x)}^2_{g(x)}
	\\
	&=
	\frac12\sn{-\frac{\alpha}{2} (\dd d^2(x, \cdot))(y)}^2_{g(y)}
	\\
	&=
	\Hh\left(y, - \frac{\alpha}{2}  (\dd d^2(x,\cdot))(y) \right).
	\end{align*}

Consequently, for $x_{\alpha,\epsilon},y_{\alpha,\epsilon}$ with $d(x_{\alpha,\epsilon},y_{\alpha,\epsilon}) < \delta$ we find
	$$
	\Hh\left(x_{\alpha,\epsilon}, \frac{\alpha}{2} (\dd d^2(\cdot, y_{\alpha,\epsilon}))(x_{\alpha,\epsilon}) \right) -  \Hh\left(y_{\alpha,\epsilon}, - \frac{\alpha}{2}  (\dd d^2(x_{\alpha,\epsilon},\cdot))(y_{\alpha,\epsilon}) \right) = 0.
	$$
	By Proposition \ref{proposition:comparison_conditions_on_H}, we can conclude that $f - \lambda Hf = h$ satisfies the comparison principle.\\

	By Theorem \ref{theorem:abstract_LDP}, the measures $\{\nu_n\}_{n\geq 1}$ satisfy in $D(\bR^+,M)$ the large deviation principle with good rate function given by \eqref{eq:rfmog}. As the topology of $D(\bR^+,M)$ restricted to $C(\bR^+,M)$ reduces to the uniform topology, the same LDP holds in $C(\bR^+,M)$, concluding the proof.
\end{proof}

\subsection{Proof of Mogulskii's Theorem, Theorem \ref{theorem:Mogulskii}}\label{subsection:Mogulskiiproof}

In this section, we prove the analogue of Mogulskii's theorem for time-scaled geodesic random walks.

\begin{proof}[Proof of Theorem \ref{theorem:Mogulskii}]
	The proof is similar as the proof of Theorem \ref{theorem:Schilder}. We verify the conditions for Theorem \ref{theorem:abstract_LDP}.

	\textit{Step 1:} We start by calculating $H_n$ and its limit $H$. Observe that for every $n \in \NN$ the sequence $\{\frac1n*\Ss_k\}_{k\geq 1}$  (taking time-step size $\varepsilon_n = n^{-1}$) is a Markov chain with transition operator given by
	$$
	T_nf(x) = \EE\left(f\left(\frac1n*\Ss_{k+1}\right)\middle|\frac1n*\Ss_k = x\right) = \int_{T_xM} f(\Exp_x(n^{-1}v)) \mu_x(\dd v). 
	$$
	Using this, for every $n \in \NN$ we can compute the Hamiltonian
	$$
	H_nf(x) 
	= 
	\log e^{-nf}T_ne^{nf}(x)
	=
	\log \int_{T_xM} e^{n(f(\exp_x(n^{-1}v) - f(x))}\mu_x(\dd v).
	$$
	Consequently, as $\Lambda_x(p) < \infty$ for all $x \in M$ and $p \in T_x^*M$, we find for $f \in C^\infty_c(M)$ that
	\begin{equation*} 
	Hf(x) := \lim_{n\to\infty} H_nf(x) = \log\int_{T_xM} e^{\ip{v}{\dd f(x)}} \mu_x(\dd v)
	\end{equation*}
	uniformly in $x$, so that we can take $\Dd(H) = f \in C^\infty_c(M)$. Note that indeed $H$ has the form $Hf(x) = \cH(x,\dd f(x))$ for a continuous map $\cH: T^*M \rightarrow \bR$ that is convex in the second coordinate. This implies that Assumption \ref{assumption:Hamiltonian_on_cotangentbundle} is satisfied and $\cH$ is given by 
	\begin{equation}\label{eqn:form_of_H_Mogulskii}
	\cH(x,p) = \log \int_{T_xM} e^{\inp{v}{p}}\mu_x(\dd v) = \Lambda_x(p).
	\end{equation}
	
	\textit{Step 2:} By Lemma \ref{lemma:goodcontainment} we have a good containment function $\Upsilon$.

	\textit{Step 3:} Fix $\lambda > 0$ and $h \in C_b(M)$. We verify the comparison principle for $f - \lambda Hf = h$ by the application of Proposition \ref{proposition:comparison_conditions_on_H}.
	Let $x_{\alpha,\epsilon},y_{\alpha,\epsilon}$ be as in Proposition\ref{proposition:comparison_conditions_on_H}. We establish \eqref{condH:negative:liminf}. 
	
	Fix $\varepsilon > 0$. By Lemma \ref{lemma:doubling_lemma}, there is a compact $K^\epsilon \subseteq M$ such that  $\{x_{\alpha,\epsilon},y_{\alpha,\epsilon}| \alpha > 0\}$ is contained in $K^\epsilon$. By the continuity of the injectivity radius and the compactness of $K^\epsilon$, we can find a $\delta > 0$ such that $i(K^\epsilon) \geq \delta > 0$. Now for $x,y \in K^\epsilon$ with $d(x,y) < \delta$ we find
	\begin{align*}
	\cH\left(x, \frac{\alpha}{2} (\dd d^2(\cdot, y))(x) \right)
	&=
	\Lambda_x\left(\frac{\alpha}{2} (\dd d^2(\cdot, y))(x) \right)
	\\
	&=
	\Lambda_y\left(\frac{\alpha}{2} \tau_{xy}(\dd d^2(\cdot, y))(x) \right)
	\\
	&=
	\Lambda_y\left(-\frac{\alpha}{2}(\dd d^2(x, \cdot))(y) \right)
	\\
	&= 
	\cH\left(y, -\frac{\alpha}{2} (\dd d^2(x, \cdot))(y) \right).
	\end{align*}
	Here $\tau_{xy}$ denotes parallel transport along the unique geodesic of minimal length connecting $x$ and $y$. The second equality follows from proposition \ref{prop:logmgf} and the third from Proposition \ref{proposition:transport_derivative_distance}. We thus find for $x_{\alpha,\epsilon},y_{\alpha,\epsilon}$ with $d(x_{\alpha,\epsilon},y_{\alpha,\epsilon}) < \delta$ that
	$$
	\Hh\left(x_{\alpha,\epsilon}, \frac{\alpha}{2} (\dd d^2(\cdot, y_{\alpha,\epsilon}))(x_{\alpha,\epsilon}) \right) - \Hh\left(y_{\alpha,\epsilon}, -\frac{\alpha}{2} (\dd d^2(x_{\alpha,\epsilon}, \cdot))(y_{\alpha,\epsilon}) \right) = 0.
	$$
	Consequently, by Proposition \ref{proposition:comparison_conditions_on_H} we find that $H$ satisfies the comparison principle.
	
	By Proposition \ref{prop:logmgf_cont_diffb} $\cH$ is continuously differentiable and hence Theorem \ref{theorem:abstract_LDP} implies that the measures $\{\nu_n\}_{n\geq 1}$ satisfy in $D(\bR^+,M)$ the large deviation principle with good rate function given by \eqref{eq:rfmog}. 
	
\end{proof}

\subsection{Proof of Cram\'er's Theorem, Theorem \ref{theorem:Cramer}}\label{subsection:Cramer}

We can now obtain the analogue of Cram\'er's theorem for Riemannian manifolds from Mogulskii's theorem via the contraction principle.


\begin{proof}[Proof of Theorem \ref{theorem:Cramer}]
Let $\tilde\nu_n$ be the measures of $Z_n(t) = \frac1n*\Ss_{\lfloor nt\rfloor}$ in $D(\RR^+;M)$. By Theorem \ref{theorem:Mogulskii}, we know that $\{\tilde\nu_n\}_{n\geq 1}$ satisfies in $D(\RR^+;M)$ the LDP with good
rate function $\tilde I$ given by \eqref{eq:rfmog}. Define $f:D(\RR^+;M) \to M$ by $f(\gamma) = \gamma(1)$. Then $f$ is continuous on $C(\RR^+,M) \subseteq D(\RR^+,M)$ and for every $n \in \NN$ we have $\nu_n = \tilde\nu_n \circ f^{-1}$. As the rate function in Mogulskii's theorem is finite only for continuous paths, we can apply the contraction principle and obtain that $\{\nu_n\}$ satisfies in $M$ the LDP with good rate function
	$$
	\tilde I_M(x) = \inf\{\tilde I(\gamma)|f(\gamma) = x\}.
	$$
	It remains to show that $\tilde I_M = I_M$. For this it is sufficient to show that we only need to consider geodesics between $x_0$ and $x$ when taking the infimum in the definition of $\tilde I_M$. To this end, let $U$ be uniformly distributed over $[0,1]$. We can write
	\begin{align*}
	\tilde I(\gamma) 
	&= 
	\int_0^1 \Lambda_{\mu_{x_0}}^*(\tau_{0\gamma(t)}^{-1}\dot\gamma(t))\ud t\\
	&= 
	\EE(\Lambda_{\mu_{x_0}}^*(\tau_{0\gamma(U)}^{-1}\dot\gamma(U)))\\
	&\geq
	\Lambda_{\mu_{x_0}}^*(\EE(\tau_{0\gamma(U)}^{-1}\dot\gamma(U))),
	\end{align*}
	
	where the last line follows from Jensen's inequality. This gives us a lower bound for $\tilde I(\gamma)$. The lower bound is achieved when 
	$$
	\tau_{0\gamma(U)}^{-1}\dot\gamma(U) = \EE(\tau_{0\gamma(U)}^{-1}\dot\gamma(U))
	$$
	almost surely. This holds if $\tau_{0\gamma(U)}^{-1}\dot\gamma(U)$ is constant, i.e. if $\gamma$ is a geodesic. This shows that for every curve $\gamma$ there exists a geodesic $\gamma'$ such that $\tilde I(\gamma') \leq \tilde I(\gamma)$. We obtain that $\tilde I_M$ is given by
	$$
	\tilde I_M(x) = \inf\left\{\int_0^1 \Lambda_{\mu_{\gamma(t)}}^*(\dot\gamma(t))\ud t\middle|\gamma:[0,1]\to M \mbox{ geodesic}, \gamma(0) = x_0, \gamma(1) = x\right\}.
	$$
	Now observe that if $\gamma$ is a geodesic, then $\dot\gamma$ is parallel along $\gamma$. Proposition \ref{prop:logmgf} implies that for all $t \in [0,1]$ 
	$$
	\Lambda_{\mu_{\gamma(t)}}^*(\dot\gamma(t)) = \Lambda_{\mu_{\gamma(0)}}^*(\dot\gamma(0))
	$$

from which it follows that 
	$$
	\tilde I_M(x) = \inf\left\{\Lambda_{\mu_{x_0}}^*(\dot\gamma(0))\middle|\gamma:[0,1]\to M \mbox{ geodesic}, \gamma(0) = x_0, \gamma(1) = x\right\}.
	$$

It remains to show that the infimum is actually attained. First note that by completeness of the manifold, for every $x$ the set is nonempty. Now if $\tilde I_M(x) = \infty$, then $\Lambda_{\mu_{x_0}}^*(\dot\gamma(0)) = \infty$ for all geodesics $\gamma:[0,1] \to M$ with $\gamma(0) = x_0$ and $\gamma(1) = x$.
Consequently, the infimum is indeed attained.

For the case when $I_M(x) < \infty$, observe that
$$
\Lambda_{\mu_{x_0}}^*(v) \to \infty \mbox{ as } |v|_{g(x_0)} \to \infty.
$$
Hence, there exists a constant $R > 0$ such that $\Lambda_{\mu_{x_0}}^*(v) > \tilde I_M(x) + 1$ whenever $||v||_{g(x_0)} > R$. As $\Exp_{x_0}$ is a continuous map, we have that $\Exp_{x_0}^{-1}(x)$ is closed, and consequently, the set 
$$
\Exp_{x_0}^{-1}(x) \cap \{v \in T_{x_0}M| |v|_{g(x_0)} \leq R\}
$$
is compact. Because $\Lambda_{\mu_{x_0}}^*$ is lower-semicontinuous, it attains its infimum on this compact set. By definition of $R$, we find that $I_M = \tilde I_M$, concluding the proof.

\end{proof}

\appendix

\section{Appendix: Stochastic differential equations on manifolds}\label{section:SDEmanifold}

In this section we will introduce the general theory of stochastic differential equations on manifolds, in which processes can have a finite explosion time. In the given paper, assumptions we make on the geometry assure that the explosion times of the processes we consider are almost surely infinite.\\

We first define what we mean by an $M$-valued semimartingale.

\begin{definition}
Let $M$ be a differentiable manifold, $(\Omega,\Ff,(\Ff_t),\PP)$ a filtered probability space and $\tau$ a stopping time with respect to the filtration $(\Ff_t)$. An \emph{$M$-valued semimartingale} is a continuous $M$-valued process $X$ on $[0,\tau)$ such that $f(X)$ is a real-valued semimartingale on $[0,\tau)$ for all $f \in C^\infty(M)$.
\end{definition} 

$M$-valued semimartingales will serve as the solution of SDEs on $M$, which we will define next. Let $V_1,\ldots,V_l$ be $l$ vector fields on $M$ and let $Z$ be an $\RR^l$-valued semimartingale. Pick an initial value $X_0 \in \Ff_0$, which is an $M$-valued random variable. We consider the stochastic differential equation 
\begin{equation}\label{eq:SDE}
\dd X_t = V_\alpha(X_t)\circ\dd Z_t^\alpha
\end{equation}
and refer to it as $SDE(V_1,\ldots,V_l;Z,X_0)$. 

As in the case of $\RR^k$-valued processes, one can use It\^o's formula to find that
$$
\dd f(X_t) = V_\alpha f(X_t) \circ \dd Z_t^\alpha
$$
for all $f \in C^\infty(\RR^d)$. Inspired by this, we give the following definition of a solution of \eqref{eq:SDE}.

\begin{definition}
An $M$-valued semimartingale $X$ defined up to a stopping time $\tau$ is a solution of \eqref{eq:SDE} up to time $\tau$ if 
\begin{equation}\label{eq:SDEdef}
f(X_t) = f(X_0) + \int_0^t V_\alpha f(X_s)\circ \dd Z_s^\alpha
\end{equation}
for all $0 \leq t \leq \tau$ and for all $f \in C^\infty(M)$.
\end{definition}

A typical approach in studying solutions of SDEs on manifolds is to embed the manifold into some Euclidean space and study a related SDE defined there. Therefore we need to know how solutions of SDEs behave under diffeormorphisms. Given a diffeomorphism $\phi:M \to N$, we can define the \emph{push-forward} $\phi_*:\Gamma(TM) \to \Gamma(TN)$ by
$$
(\phi_*V)f(\phi(x)) = V(f\circ\phi)(x)
$$
where $f \in C^\infty(N)$ and $V \in \Gamma(TM)$. The push-forward is also referred to as the differential of a function between manifolds, and is in that case denoted by $\dd\phi$ rather than $\phi_*$.  The following proposition, which is Proposition 1.2.4 in \cite{Hsu02}, shows that solutions of SDEs behave nicely under diffeomorphisms.
\begin{proposition}\label{prop:difSDE}
Let $\phi:M \to N$ be a diffeomorphism and suppose that $X$ is a solution of
$$
\dd X_t = V_\alpha \circ \dd Z_t^\alpha
$$
on $M$ with given initial value $X_0$. Then $Y = \phi(X)$ is a solution of
$$
\dd Y_t = \phi_*V_\alpha \circ \dd Z_t^\alpha
$$
on $N$ with given initial value $Y_0 = \phi(X_0)$.
\end{proposition}


\section{Appendix: Orthonormal frame bundles}\label{section:framebundle}

In this section we introduce the orthonormal frame bundle. We follow the approach in \cite[Section 2.1]{Hsu02}, restricting to orthonormal frame bundles rather than general frame bundles.

An \emph{orthonormal frame} at a point $x \in M$ is an isometry $u:\RR^k \to T_xM$. Denoting $e_1,\ldots,e_k$ the standard basis of $\RR^k$, the tangent vectors $ue_1,\ldots,ue_k$ form a basis for $T_xM$. We will denote the space of all possible orthonormal frames at $x$ by $OM_x$. This can be made into a bundle $OM = \bigsqcup_{x\in M} OM_x$, which we will refer to as the \emph{orthonormal frame bundle}. This can itself be made into a differentiable manifold of dimension $k + \frac12k(k-1)$ such that the projection $\pi:OM \to M$ is a smooth map.

\subsection{Vertical and horizontal tangent vectors}

A tangent vector $X \in T_uOM$ is called \emph{vertical} if it is tangent to $OM_{\pi u}$. This means that $X$ is the tangent vector of a curve through $u$ which remains in $\{\pi u\} \times OM_{\pi u}$. We denote the space of vertical tangent vectors at $u$ by $V_uOM$, which is a $\frac12k(k-1)$-dimensional subspace of $T_uOM$. Note that the notion of verticality is independent of the connection on the manifold.\\ 

Now suppose that $M$ is supplied with the Levi-Civita connection corresponding to the metric. This allows us to define horizontal tangent vectors as well. We say that a curve $(u_t)_t$ in $OM$ is \emph{horizontal} if for all $e \in \RR^d$ it holds that $(u_te)_t$ is parallel along $(\pi u_t)_t$. A tangent vector $X \in T_uOM$ is now called \emph{horizontal} if it is tangent to a horizontal curve through $u$. We denote the space of horizontal tangent vectors at $u$ by $H_uOM$, which is a $k$-dimensional subspace of $T_uOM$. Additionally, we have that
$$
T_uOM = H_uOM \oplus V_uOM.
$$
The projection $\pi:OM \to M$ gives rise to a homomorphism $\pi_*:T_uOM \to T_{\pi u}M$, which is simply the push-forward of tangent vectors. Its kernel turns out to be $V_uOM$, in which case it induces an isomorphism $\pi_*:H_uOM \to T_{\pi u}M$. This means that for any $X \in T_xM$ and orthonormal frame $u$ at $x$ there exists a unique $X^* \in H_uOM$ such that $\pi_*X^* = X$. We call $X^*$ the \emph{horizontal lift} of $X$ to $u$.

\subsection{Horizontal lift and anti-development}

Given a curve $(x_t)_t$ in $M$, we want to identify it with a curve in $\RR^d$ in a suitable way. This curve will be referred to as the anti-development of $(x_t)_t$. In order to do this, we first need to go via the orthonormal frame bundle. 

Given an initial frame $u_0$ at $x_0$, there exists a unique horizontal curve $(u_t)_t$ in $OM$ such that $\pi u_t = x_t$. We call this the \emph{horizontal lift} of $(x_t)_t$ via $u_0$. It gives rise to a linear map
$$
\tau_{t_0t_1} = u_{t_1}u_{t_0}^{-1}: T_{x_{t_0}}M \to T_{x_{t_1}}M
$$
which is independent of $u_0$ and is referred to as parallel transport along $(x_t)_t$.

The idea of the horizontal lift of a curve is that the coordinates of the parallel transport of a tangent vector remain ''the same''. What is meant by this, is that if $e \in \RR^k$ are the coordinates of a tangent vector $X \in T_{x_{t_0}}M$ with respect to the frame $u_{t_0}$, then they are also the coordinates of its parallel transport $\tau_{t_0t_1}X \in T_{x_{t_1}}M$ with respect to the frame $u_{t_1}$.\\





Using a horizontal lift $(u_t)_t$ of the curve $(x_t)_t$, we can define its anti-development to $\RR^k$. To do this, one first observes that $u_t^{-1}\dot x_t \in \RR^k$. Consequently, we may define
$$
w_t = \int_0^t u_s^{-1}\dot x_s \ud s,
$$
which is a curve in $\RR^k$, called the \emph{anti-development} of $(x_t)_t$. Note that this curve depends on the chosen initial frame, although in a fairly simple way. Indeed, if we choose another initial frame $v_0$ such that $u_0 = v_0g$, then the anti-development becomes $(gw_t)_t$. 

It turns out that we can connect the anti-development and horizontal lift of a curve $(x_t)_t$ in $M$ via an ordinary differential equation on $OM$. It is easy to see that $u_t\dot w_t = \dot x_t$ and consequently, using the definition of horizontal lift, we find
\begin{equation}\label{eq:liftode}
H_{\dot w_t}(u_t) = (u_t\dot w_t)^* = (\dot x_t)^* = \dot u_t.
\end{equation}

We can rewrite \eqref{eq:liftode} by writing $H_{\dot w_t}$ in terms of a basis for $H_{u_t}OM$. We will construct a specific set of basis vectors. Let $\{e_1,\ldots,e_k\}$ be the standard basis of $\RR^k$. For $i = 1,\ldots,k$, define the horizontal vector field $H_i$ by
$$
H_i(u) = (ue_i)^*,
$$
the horizontal lift of the tangent vector $ue_i$. The vector fields $H_1,\ldots,H_k$ are called the \emph{fundamental horizontal vector fields} and it holds that the vectors $\{H_1(u),\ldots,H_k(u)\}$ span $H_uOM$ for any $u \in OM$.



Using these fundamental horizontal vector fields, we may also write \eqref{eq:liftode} as
$$
\dot u_t = H_i(u_t)\dot w_t^i.\\
$$

Equation \eqref{eq:liftode} also allows us to go from a curve $(w_t)_t$ in $\RR^k$ to a curve $(x_t)_t$ in $M$, which will be called the \emph{development} of $(w_t)_t$ to $M$. Indeed, starting from a curve $(w_t)_t$ in $\RR^d$, the unique solution to \eqref{eq:liftode} is a horizontal curve $(u_t)_t$ in $OM$. Projecting it to $M$ then gives us a curve $(x_t)_t$ in $M$. This procedure is often referred to as 'rolling without slipping'.\\


\section{Appendix: Proofs of certain propositions and lemmas}\label{section:proofs_propositions_lemmas}

In this appendix we collect the proofs of several propositions and lemmas used throughout the paper. 

\subsection{Proof of Proposition \ref{proposition:transport_derivative_distance}}

For a path $h:[0,1]\to M$, define the Lagrangian
$$
L(h(t)) = \inp{\dot h(t)}{\dot h(t)}_{g(h(t))} = |\dot h(t)|^2_{g(h(t))}
$$
and the action
$$
S(h) = \int_0^1 L(h(t))\ud t.
$$

Observe that for $x, y \in M$ we have
$$
d^2(x,y) = \inf\{S(h)| h(0) = x, h(1) = y, h \mbox{ piecewise smooth}\}.
$$

If $y \notin C_x$, there is an optimal path $\gamma:[0,1] \to M$ for $S$, the geodesic of minimal length connecting $x$ and $y$. Note that the differential of the action in the starting point equals the momentum of the optimal path $\gamma$ in 0 (see e.g. \cite[Chapter 3]{Arn89}). In coordinates one finds that the $j$-th component of this momentum equals
$$
p_j = \frac{\partial L}{\partial \dot h^j(t)}(\gamma(t)) = 2\sum_{i=1}^k g_{ij}(\gamma(t))\dot\gamma^i(t) = 2(\dot\gamma(t))^\#_j
$$
where $V^\#$ denotes the covector dual to $V$. 

Consequently, we find that $\dd_xd^2(x,y) = 2(\dot \gamma(0))^\#$, where $\gamma$ is the geodesic of minimum length connecting $x$ and $y$. In particular, this implies that for any $V \in T_xM$ we have $\dd_xd^2(x,y)(V) = 2\inp{V}{\dot\gamma(0)}_{g(x)}$. Defining $\gamma^-(t) := \gamma(1 - t)$, $\gamma^-$ is the geodesic of minimum length connecting $y$ and $x$. We obtain $\dd_yd^2(x,y) = 2(\dot \gamma^-(0))^\# = -2(\dot\gamma(1))^\#$. Noticing that $\dot\gamma(1)$ is the parallel transport of $\dot\gamma(0)$ now proves the claim.

\subsection{Proof of Theorem \ref{thm:FW}}

We can rewrite SDE \eqref{eq:SDEFW} to obtain the It\^o SDE
\begin{equation}\label{eq:SDEFW1}
\dd Y_t^\epsilon = b(Y_t^\epsilon)\ud t + \frac12\epsilon (D\sigma\cdot\sigma)(Y_t^\epsilon)\ud t + \sqrt{\epsilon}\sigma(Y_t^\epsilon)\ud W_t,
\end{equation}
where $D\sigma$ denotes the total differential of $\sigma$. 

Now suppose that $\tilde Y_t^\epsilon$ satisfies the It\^o SDE
\begin{equation}\label{eq:SDEFW2}
\dd \tilde Y_t^\epsilon = b(\tilde Y_t^\epsilon)\ud t + \sqrt{\epsilon}\sigma(\tilde Y_t^\epsilon)\ud W_t, \qquad \tilde Y_0^\epsilon \stackrel{D}{=} Y_0^\epsilon.
\end{equation}
By Theorem \ref{theorem:FWIto}, we find that $\tilde Y_t^\epsilon$ satisfies the LDP in $C(\RR^+;\RR^k)$ with the good rate function $I$ as in \eqref{equation:RFFW}. To complete the proof, it suffices to show that $Y_t^\epsilon$ and $\tilde Y_t^\epsilon$ are exponentially equivalent in $C([0,T];\RR^k)$ for all $T > 0$.\\

Fix $T > 0$. Consider the joint law of $Y_0^\epsilon$ and $\tilde Y_0^\epsilon$ obtained by setting $Y_0^\epsilon = \tilde Y_0^\epsilon$. Consider the following system of stochastic differential equations
$$
\begin{cases}
\dd Y_t^\epsilon = b(Y_t^\epsilon)\ud t + \frac12\epsilon(D\sigma\cdot\sigma)(Y_t^\epsilon)\ud t + \sqrt{\epsilon}\sigma(Y_t^\epsilon)\ud W_t\\
\dd \tilde Y_t^\epsilon = b(\tilde Y_t^\epsilon)\ud t + \sqrt{\epsilon}\sigma(\tilde Y_t^\epsilon)\ud W_t\\
\dd Z_t^\epsilon = b(Y_t^\epsilon)\ud t + \sqrt{\epsilon}\sigma(Y_t^\epsilon)\ud W_t,\\
\end{cases}
$$
where $Y_0^\epsilon$ has some given distribution and $Z_0^\epsilon = \tilde Y_0^\epsilon = Y_0^\epsilon$.

First note that
$$
\dd(Z_t^\epsilon - Y_t^\epsilon) = -\frac12\epsilon(D\sigma\cdot\sigma)(Y_t^\epsilon)\ud t.
$$
As $\sigma$ is Lipschitz continuous, $D\sigma$ is bounded, which together with the boundedness of $\sigma$ implies that $|Z_t^\epsilon - Y_t^\epsilon| \leq CT\epsilon$ for some constant $C$ which only depends on the bound and Lipschitz constant of $\sigma$ and the dimension $k$.

Furthermore, we have
$$
\dd(\tilde Y_t^\epsilon - Z_t^\epsilon) = (b(\tilde Y_t^\epsilon) - b(Y_t^\epsilon)) \dd t + \sqrt\epsilon(\sigma(\tilde Y_t^\epsilon) - \sigma(Y_t^\epsilon))\ud W_t.
$$
By the estimate above, and the Lipschitz continuity of $b$ we find that
\begin{align*}
|b(\tilde Y_t^\epsilon) - b(Y_t^\epsilon)| 	&\leq B|\tilde Y_t^\epsilon - Y_t^\epsilon|\\
							&\leq B(|\tilde Y_t^\epsilon - Z_t^\epsilon| + |Z_t^\epsilon - Y_t^\epsilon|)\\
							&\leq B\sqrt2(|\tilde Y_t^\epsilon - Z_t^\epsilon|^2 + |Z_t^\epsilon - Y_t^\epsilon|^2)^{1/2}\\
							&\leq B\sqrt2(|\tilde Y_t^\epsilon - Z_t^\epsilon|^2 + C^2T^2\epsilon^2)^{1/2}.
\end{align*}
A similar estimate holds with $\sigma$ instead of $b$. Consequently, 
by Lemma 5.6.18 in \cite{DZ98} we find for $\delta > 0$ that
$$
\limsup_{\epsilon \to 0} \epsilon \log P(\sup_{0 \leq t \leq T}|\tilde Y_t^\epsilon - Z_t^\epsilon| \geq \delta) \leq \limsup_{\epsilon \to 0} K + \log\left(\frac{C^2\epsilon^2}{C^2\epsilon^2 + \delta^2}\right) = -\infty.
$$
This shows that $\tilde Y_t^\epsilon$ and $Z_t^\epsilon$ are exponentially equivalent. Because $Y_t^\epsilon$ and $Z_t^\epsilon$ are clearly also exponentially equivalent, we conclude that $Y_t^\epsilon$ and $\tilde Y_t^\epsilon$ are exponentially equivalent as desired.

\subsection{Proof of Proposition \ref{prop:stopping}}

The proof is based on the proof of Theorem 3.6.1 in \cite{Hsu02}. Define the radial process $R_t = d(W_t,x_0)$. There exists a standard Eucledian Brownian Motion $B_t$ such that
$$
R_t = B_t + \frac12\int_0^t\Delta_MR(W_s)\ud s - L_t
$$
where $L_t$ is a nondecreasing process which only increases on the cutlocus of $x_0$.

By It\^o's formula we have
$$
R_t^2 = 2\int_0^t R_s\ud R_s + [R,R]_t.
$$

Using the expression for $R_t$, we immediately see that
$$
[R,R]_t = [B,B]_t = t,
$$
where the latter holds as $B$ is standard Euclidean Brownian motion.

Remembering that $L_t$ is nondecreasing, we obtain
$$
R_t^2 \leq 2\int_0^tR_s\ud B_s + \int_0^tR_s\Delta_M R(W_s)\ud s + t.
$$
By the Laplacian comparison theorem (see e.g. \cite[Corollary 3.4.4]{Hsu02}), we obtain, using the lower bound on the Ricci-curvature, that
$$
R_t\Delta_MR_t \leq (k-1)L\coth LR_t \leq (k-1)(1 + L).
$$

Now let $T_\delta$ be the first exit time of $W_t$ from the geodesic ball $B(x_0,\delta)$. Combining the inequalities, we obtain
$$
\delta^2 \leq 2\int_0^{T_\delta} R_s\ud B_s + (k-1)(1 + L)T_\delta + t \leq 2\int_0^{T_\delta} R_s\ud B_s + 2kLT_\delta.
$$

Now on the set $\{T_\delta \leq \tau\}$ we have
$$
\int_0^{T_\delta} R_s\ud B_s \geq \frac12\delta^2 - kL\tau.
$$

As $\int_0^{T_\delta} R_s\ud B_s$ is a local martingale, we know that it is a continuous time-change of Brownian motion. More specifically, there exists a Brownian motion $\tilde B$ such that
$$
\int_0^{T_\delta} R_s\ud B_s = \tilde B_\eta,
$$
where
$$
\eta = \int_0^{T_\delta} R_s^2\ud s.
$$

On the set $\{T_\delta \leq \tau\}$ it holds that
$$
\eta \leq \delta^2T_\delta \leq \delta^2\tau.
$$

Consequently,
$$
\max_{0 \leq s \leq \delta^2\tau} \tilde B_s \geq \tilde B_\eta \geq \frac12\delta^2 - kL\tau.
$$
Now the left-hand side is distributed as $\delta\sqrt{\tau}|\tilde B_1|$. It now follows that
\begin{align*}
\PP\left(\sup_{0\leq t \leq \tau} d(W_t,x_0) \geq \delta\right)		&=\PP(T_\delta \leq \tau)\\	
				&\leq \PP(\delta\sqrt{\tau}|\tilde B_1| \geq \frac12\delta^2 - kL\tau)\\
				&= \PP\left(|\tilde B_1| \geq \frac{\frac12\delta^2 - kL\tau}{\delta\sqrt{\tau}}\right)\\
				&\leq 2e^{-\frac12\frac{(kL\tau - \frac12\delta^2)^2}{\delta^2\tau}}
\end{align*}
Here, the last estimate follows simply by the fact that $\tilde B_1$ has a standard normal distribution so that $\PP(\tilde B_1 \geq \alpha) \leq e^{-\frac12\alpha^2}$ for $\alpha > 0$.


\smallskip

\textbf{Acknowledgement}
RK was supported by the Deutsche Forschungsgemeinschaft (DFG) via RTG 2131 High-dimensional Phenomena in Probability – Fluctuations and Discontinuity. RV was supported by the Peter Paul Peterich Foundation via the TU Delft University Fund.


\bibliographystyle{abbrv} 
\bibliography{../Bibliography/VersendaalBib}{}

\end{document}